\documentclass[11pt]{article}
\usepackage{amsmath}
\usepackage[utf8]{inputenc}
\usepackage{amsthm}
\usepackage{thmtools}
\usepackage{hyperref, cleveref}
\usepackage[a4paper, total={7in, 9in}]{geometry}
\usepackage{blindtext}
\usepackage{comment}
\usepackage{geometry}
\providecommand{\keywords}[1]
{
	\small	
	\textbf{\textit{Keywords:}} #1
}
\usepackage{amsmath}
\NewDocumentEnvironment{alignb}{b}{%
	\begin{align*}
		\refstepcounter{equation} #1 \tag{\theequation}
	\end{align*}
}{}
\allowdisplaybreaks
\usepackage{enumerate}
\makeatletter
\newcommand{\myitem}[1]{%
	\item[#1]\protected@edef\@currentlabel{#1}%
}
\makeatother
\usepackage{clipboard}
\declaretheorem[numberwithin=section]{theorem, definition}
\declaretheorem{lemma, proposition, remark}[style=plain,
numberwithin=section]
\numberwithin{equation}{section}
\newtheorem{mytheorem}{Theorem}
\newtheorem{defi}{Definition}[section]
\numberwithin{mytheorem}{section}
\usepackage{amsbsy}
\usepackage{graphicx}
\usepackage[T1]{fontenc}
\usepackage{lmodern}
\usepackage{imakeidx}
\usepackage{amssymb}
\usepackage{thmtools}
\usepackage{cite}
\newenvironment{myproof}[2] {\paragraph{Proof of {#1} {#2} :}}{\hfill$\square$}

\usepackage[english]{babel}
\usepackage[usenames,dvipsnames]{color}
\title{ Normalized solutions to a quasilinear equation involving critical Sobolev exponent}
\author{Nidhi\footnote{Department of Mathematics, Indian Institute of Technology, Delhi, Hauz Khas, New Delhi-110016, India. e-mail: nidhi.nidhi@maths.iitd.ac.in}\; and K. Sreenadh\footnote{Department of Mathematics, Indian Institute of Technology, Delhi, Hauz Khas, New Delhi-110016, India. e-mail: sreenadh@maths.iitd.ac.in}\;}
\date{}
\begin{document}
	\maketitle
	\begin{abstract}
		\noindent In this paper we study the existence and regularity results of normalized solutions to the following quasilinear elliptic Choquard equation with critical Sobolev exponent and mixed diffusion type operators:
		\begin{equation*}
			\begin{array}{rcl}
				-\Delta_p	u+(-\Delta_p)^su & = & \lambda |u|^{p-2}u +|u|^{p^*-2}u+ \mu(I_{\alpha}*|u|^q)|u|^{q-2}u\;\;\text{in } \mathbb{R}^N,\\
				\int_{\mathbb{R}^N}|u|^pdx & = & \tau,
			\end{array}
		\end{equation*}
		where $N\geq 3$, $\tau>0$, $\frac{p}{2}(\frac{N+\alpha}{N})<q<\frac{p}{2}(\frac{N+\alpha}{N-p})$, $I_{\alpha}$ is the Riesz potential of order $\alpha\in (0,N)$, $\mu>0$ is a parameter, $(-\Delta_p)^s$ is the fractional p-laplacian operator, $p^*=\frac{Np}{N-p}$ is the critical Sobolev exponent and  $\lambda$ appears as a Lagrange multiplier.\\
		\noindent \keywords{Normalized solutions, Choquard equation, critical growth, local and nonlocal operator, existence results, H$\ddot{\text{o}}$lder regularity. }
	\end{abstract}
	\section{Introduction}
	This work is mainly focused on the existence and regularity results of the following quasilinear equation with critical Sobolev exponent, Choquard nonlinearity and a mixed diffusion type operator:
	\begin{equation}\label{1.1}
		-\Delta_p	u+(-\Delta_p)^su  =  \lambda |u|^{p-2}u +|u|^{p^*-2}u+ \mu(I_{\alpha}*|u|^q)|u|^{q-2}u\;\;\text{in } \mathbb{R}^N,
	\end{equation}
	with a fixed $L^p$ norm
	\begin{equation}\label{1.2}
		\int_{\mathbb{R}^N}|u|^pdx  =  \tau,
	\end{equation}
	where $N\geq 3$, $0<s<1<p<N$, $\tau>0$, $p^*=\frac{Np}{N-p}$ is the critical Sobolev exponent.
	The operators p-laplacian ($\Delta_p$) and fractional p-laplacian $(-\Delta_p)^s$ are defined as:
	$$\Delta_pu=div(|\nabla u|^{p-2}\nabla u),$$
	and
	$$(-\Delta_p)^su(x)=\lim_{\epsilon\rightarrow 0}\int_{\mathbb{R}^N\setminus B_{\epsilon}(0)}\frac{|u(x)-u(y)|^{p-2}(u(x)-u(y))}{|x-y|^{N+sp}}dy \text{ for } s\in (0,1).$$
	$I_{\alpha}$ is the Riesz potential of order $\alpha\in (0,N)$ given by
	\begin{equation}\label{A_alpha}
		I_{\alpha}(x)=\frac{A_{N,\alpha}}{|x|^{N-\alpha}}\text{ with } A_{N,\alpha}=\frac{\Gamma(\frac{N-2}{2})}{\pi^{\frac{N}{2}}2^{\alpha}\Gamma(\frac{\alpha}{2})}\text{ for every }x\in\mathbb{R}^N\setminus \{0\}.
	\end{equation}
	The problem is well defined for all $\frac{p}{2}(\frac{N+\alpha}{N})\leq q\leq \frac{p}{2}(\frac{N+\alpha}{N-p})$ due to the following well known Hardy-Littlewood-Sobolev inequality:
	\begin{proposition}\label{prop1.1}
		Let $t,r>1$ and $0<\alpha <N$ with $1/t+1/r=1+\alpha/N$, $f\in L^t(\mathbb{R}^N)$ and $h\in L^r(\mathbb{R}^N)$. There exists a sharp constant $C(t,r,\alpha,N)$ independent of $f,h$, such that
		\begin{equation}\label{HLS}
			\int_{\mathbb{R}^N}\int_{\mathbb{R}^N}\frac{f(x)h(y)}{|x-y|^{N-\alpha}}~dxdy \leq C(t,r,\alpha,N) \|f\|_{L^t}\|h\|_{L^r}.
		\end{equation}
		If $t=r=2N/(N+\alpha)$, then
		\begin{align}\label{C_alpha}
			C(t,r,\alpha,N)=C(N,\alpha)= \pi^{\frac{N-\alpha}{2}}\frac{\Gamma(\frac{\alpha}{2})}{\Gamma(\frac{N+\alpha}{2})}\left\lbrace \frac{\Gamma(\frac{N}{2})}{\Gamma(N)}\right\rbrace^{-\frac{\alpha}{N}}.
		\end{align}
		Equality holds in  \eqref{HLS} if and only if $\frac{f}{h}\equiv constant$ and
		$\displaystyle h(x)= A(\gamma^2+|x-a|^2)^{(N+\alpha)/2}$
		for some $A\in \mathbb{C}, 0\neq \gamma \in \mathbb{R}$ and $a \in \mathbb{R}^N$.
	\end{proposition}
	\noindent It follows from \autoref{prop1.1} that
	\begin{align*}
		{\mathcal{A}_q(u):=	\int_{\mathbb R^N}\int_{\mathbb R^N}\frac{|u(x)|^{q}|u(y)|^{q}}{|x-y|^{N-\alpha}}~dxdy}
	\end{align*}
	is well defined if $\frac{p}{2}\left(\frac{N+\alpha}{N}\right)\leq q \leq \frac{p}{2}\left(\frac{N+\alpha}{N-p}\right)$.
	The expression $(I_{\alpha}*|u|^q)|u|^{q-2}u$ is referred to as Choquard-type nonlinearity, as it was employed by Choquard in the examination of the Hartree-Fock theory of a one-component plasma. At the Symposium on Coulomb Systems, Choquard explained how the energy functional associated with the problem:
	\begin{equation}\label{Choq}
		\left\{ \begin{array}{rl}    & -\Delta u  +  u = (I_2*|u|^2)u\;\;\text{in } \mathbb{R}^3,\\ &  u \in H^1(\mathbb{R}^3). 
		\end{array} \right.
	\end{equation}
	can be used in this regard, see \cite{lieb1977existence}. Also, \eqref{Choq} was used by Pekar in \cite{pekar1954untersuchungen} to study the quantum theory model of a stationary polaron. The equations of the type 
	\begin{equation}\label{Choq_eq}
		-\Delta u +\lambda u = \mu(I_{\alpha}*|u|^p)|u|^{p-2}u \text{ in } \mathbb{R}^N,
	\end{equation}
	are called the Choquard equation. Such equations mostly occur in the field of quantum physics, see \cite{penrose1996gravity}, and have been extensively studied. We refer the readers to go through \cite{filippucci2020singular}, \cite{liu2022another}, and \cite{moroz2013groundstates} to see how the existence, multiplicity, and qualitative characteristics of the solution to \eqref{Choq_eq} are deduced.
\noindent This study is devoted to the normalized solution for a category of quasilinear elliptic Choquard equation associated with the operator $\mathcal{L}=-\Delta_p+(-\Delta_p)^s$. To the best of our knowledge, there hasn't been much research on such problems. Formally speaking, the solution to the class of problems of the form: 
\begin{equation}\label{Norm_sol}
	\left\{ \begin{array}{rl}   
		& -\Delta u  =  \lambda u +g(u)\;\;\text{in } \mathbb{R}^N,\\
		&  \int_{\mathbb{R}^N}|u|^2dx  =  c.
	\end{array}
	\right.
\end{equation}
is called a normalized solution. The solution of \eqref{Norm_sol} gives us the stationary state of a nonlinear Schr$\ddot{\text{o}}$dinger equation with a predetermined fixed $L^2$-norm. Jeanjean studied \eqref{Norm_sol} in \cite{Jeanjean1997Existence} and deduced the existence of a radially symmetric solution under some assumptions on $g$. Further, Bartsch and De Valeriola proved the existence of infinitely many solutions (see \cite{bartsch2012normalized}). Moreover, taking $g(t)=|t|^{p-1}t$, \eqref{Norm_sol} has been studied for bounded domains along with some boundary conditions; interested readers can go through \cite{noris2015existence} and \cite{pierotti2017normalized}. The existence of normalized solutions for nonlinear Schrödinger systems has been thoroughly investigated. Readers seeking further information may consult the references \cite{gou2018multiple, bartsch2016normalized, bartsch2018normalized, bartsch2019multiple, noris2014stable, noris2019normalized}.
The study of quadratic ergodic mean field games systems also investigates normalized solution types, as discussed in \cite{pellacci2021normalized}. \\
\noindent Recently, the study of normalized solutions to the Choquard equation has attracted the researchers. The linear case, that is $p=2$, has been dealt with by a lot of authors; for instance, see \cite{lei2023sufficient,Shen2024Normalized,Soave2020Normalized}.
The authors in \cite{Shang2023Normalized} and \cite{Meng2024Normalized} have studied the existence of normalized solutions to the critical Choquard equation, where the critical exponent is due to the above-mentioned Hardy-Littlewood-Sobolev inequality, involving classical and fractional Laplacian operators, respectively and in \cite{Giacomoni2024Normalized}, such critical growth Choquard equation has been seen studied with the presence of mixed diffusion type operator.\\
Motivated by the previously mentioned literature, we aim to investigate the existence and regularity results for normalized solutions to a broader category of quasilinear elliptic equations featuring critical Sobolev exponents and mixed diffusion-type operators. In addition to being a critical case, the presence of classical p-laplacian and fractional p-laplacian operators in the equation intensifies the complexity and interest of our study. \\
We will initiate our study by discussing the regularity properties of a radially symmetric positive weak solution of \eqref{1.1}-\eqref{1.2}.
Here, the space framework is the Banach space $W^{1,p}(\mathbb{R}^N)$ equipped with the following equivalent norm:
	$$\left\| u\right\|=(\left\| \nabla u\right\|_p^p+\left\| u\right\|_p^p+[u]_{s,p}^p)^{\frac{1}{p}},$$
	where  
	$$[u]_{s,p}^p:=\int_{\mathbb{R}^N}\int_{\mathbb{R}^N}\frac{|u(x)-u(y)|^p}{|x-y|^{N+sp}}dxdy,$$
	and $W^{1,p}_r(\mathbb{R}^N)=\{u\in W^{1,p}(\mathbb{R}^N): u \text{ is radial}\}$. Morever, the space $\mathbb{W}=W^{1,p}(\mathbb{R}^N)\times \mathbb{R}$, with $\left\| (u,t)\right\|^p_{\mathbb{W}}:=\left\| u\right\|^p+|t|^p$, will also be used. The notion of weak solution for \eqref{1.1}-\eqref{1.2} is as follows:
	\begin{defi}
		A function $u\in W^{1,p}(\mathbb{R}^N)$ is said to be the weak solution of \eqref{1.1}-\eqref{1.2} if $\left\| u\right\|_p^p =\tau$ and
		\begin{equation*}\label{weak_sol}
			\int_{\mathbb{R}^N} |\nabla u|^{p-2}\nabla u\nabla v +\ll u, v \gg_{s,p} = \lambda\int_{\mathbb{R}^N}|u|^{p-2}uv+\int_{\mathbb{R}^N}|u|^{p^*-2}uv+\mu\int_{\mathbb{R}^N}(I_{\alpha}*|u|^{q})|u|^{q-2}uv,
		\end{equation*}
		for every $v\in W^{1,p}(\mathbb{R}^N)$,
		where $\ll u,v \gg_{s,p}:=\frac{1}{2}\int_{\mathbb{R}^N}\int_{\mathbb{R}^N}\frac{|u(x)-u(y)|^{p-2}(u(x)-u(y))(v(x)-v(y))}{|x-y|^{N+2s}}dxdy$.
	\end{defi}
	\noindent Defining $A:=\mu\mathcal{A}_{q}$, the energy functional corresponding to the problem \eqref{1.1} is given by
	$$F(u)=\frac{1}{p}\left\|\nabla u \right\|_{p}^p+\frac{\lambda}{p}\left\|  u\right\|_p^p+\frac{1}{p}[u]_{s,p}^p-\frac{1}{p^*}\left\|u\right\|_{p^*}^{p^*}-\frac{1}{2q}A(u),$$
	that is, a critical point of $F$ which satisfies the constraint turns out to be the weak solution of \eqref{1.1}-\eqref{1.2}. Precisely, we can say the following about a solution of \eqref{1.1}-\eqref{1.2} :
		\begin{mytheorem}\label{Theorem 1.1}
		Let $p\geq 2$, $N-\alpha\leq 2p$ and $p\leq q\leq \frac{p}{2}(\frac{N+\alpha}{N-p})$. If $0<u\in W^{1,p}(\mathbb{R}^n)$ is a solution of \eqref{1.1}-\eqref{1.2}, corresponding to some $\lambda<0$, then $u\in L^{\infty}_{loc}(\mathbb{R}^N)$. Moreover, if $u$ is radially symmetric, then $u\in L^{\infty}(\mathbb{R}^N)\cap C^{\delta}_{loc}(\mathbb{R}^N)$ for every $0<\delta<\Theta=\min\{\frac{sp}{p-1},1\}$.
	\end{mytheorem}
		\noindent Following the work of R. Biswas and S. Tiwari in \cite{Biswas2023Regularity}, an  iterative scheme was constructed to show that a  positive solution must lie in $L^{\infty}_{loc}(\mathbb{R}^N)$. Moreover, using the radial lemma \cite{Yuan2013radial}, it is evident that a radially symmetric solution would lie in $L^{\infty}(\mathbb{R}^N)$, and hence the results of  \cite{Garain2023Higher} can be used to figure out the H$\ddot{\text{o}}$lder regularity.\\ 
	After obtaining the regularity results, we investigated the existence of the normalized solution.
	Inspired by the work of X. Shang and P. Ma in \cite{Shang2023Normalized}, existence results for the cases: $\frac{p}{2}\left(\frac{N+\alpha}{N}\right)< q < \frac{p}{2}\left(\frac{N+\alpha+sp}{N}\right)$ and $\frac{p}{2}\left(\frac{N+\alpha+sp}{N}\right)\leq q < \frac{p}{2}\left(\frac{N+\alpha}{N-p}\right)$ is discussed separately. Settting the set of functions satifying the constraint by:
	$$S(\tau):=\left\{u\in W^{1,p}({\mathbb{R}^N}):\left\| u\right\|_p^p=\tau\right\}.$$
	and using the fiber map defined as$$t\star u (x):=e^{\frac{Nt}{p}}u(e^tx),$$
	the idea is to work on the natural Pohozaev manifold, given by
	$$P(\tau):= S(\tau)\cap\bar{M} =S(\tau)\cap\left\{u\in W^{1,p}(\mathbb{R}^N): M(u)=0\right\},$$
	where $$M(u):= \left\| \nabla u \right\|_p^p+s[u]_{s,p}^p-\left\| u\right\|_{p^*}^{p^*}-\gamma_qA(u),$$
	with $\gamma_q=\frac{N}{p}-\frac{N+\alpha}{2q}$. Using the concentration compactness principles by Lions in \cite{Lions1984concentration} we deduce the following for $\frac{p}{2}\left(\frac{N+\alpha}{N}\right)<q<\bar{q}_s:=\frac{p}{2}\left(\frac{N+\alpha+sp}{N}\right)$:
	\begin{mytheorem}\label{Theorem 1.2}
		For $0<s<1<p<N $, and $\frac{p}{2}\left(\frac{N+\alpha}{N}\right)<q<\bar{q}_s$, there exists $\tau_0>0$ such that, for all $\tau<\tau_0$, the problem \eqref{1.1}-\eqref{1.2} has a radially symmetric ground state solution $u^*_{\tau}\in W^{1,p}(\mathbb{R}^N)$ corresponding to some $\lambda<0$ with $E(u^*_{\tau})<0$.
	\end{mytheorem}
	\noindent Further, for the case of $\bar{q}_s\leq q  <\frac{p}{2}\left(\frac{N+\alpha}{N-p}\right)$,
	$\epsilon$ - blow up analysis was used to conclude the following :
	\begin{mytheorem}\label{Theorem 1.3}
		Let $0<s<1<p<N$ and $\mu>0$ is sufficiently large, then we have the following:
		\begin{enumerate}
			\item if $q=\bar{q}_s$, then there exists, $\bar{\tau}_s>0$ such that for all $\tau\in (0,\bar{\tau}_s)$, \eqref{1.1}-\eqref{1.2} has a solution $(u_{\tau},\lambda_{\tau})\in W^{1,p}_r(\mathbb{R}^N)\times \mathbb{R}$ for some $\lambda<0$,
			\item  if $\bar{q}_s<q\leq\bar{q}:=\frac{p}{2}\left(\frac{N+\alpha+p}{N}\right)$, then there exists, $\bar{\tau}_q>0$ such that for all $\tau\in (0,\bar{\tau}_q)$, \eqref{1.1}-\eqref{1.2} has a solution $(u_{\tau},\lambda_{\tau})\in W^{1,p}_r(\mathbb{R}^N)\times \mathbb{R}$  for some $\lambda<0$,
			\item  for all $\bar{q}<q<\frac{p}{2}\left(\frac{N+\alpha}{N-p}\right)$, \eqref{1.1}-\eqref{1.2} has a solution $(u_\tau,\lambda_{\tau})\in W_r^{1,p}(\mathbb{R}^N)\times \mathbb{R}$ for some $\lambda<0$.
		\end{enumerate}
	\end{mytheorem}
	\subsection*{Notations}
	We have used the following notations throughout the paper:
	\begin{itemize}
		\item $T(u)=\left\| \nabla u \right\|_p^p+[u]_{s,p}^p$
		\item 
		$S$ is the best Sobolev constant, see \cite{Talenti1976}, given by:
		\begin{equation}\label{S}
			S=\inf_{u\in W^{1,p}(\mathbb{R}^N)\setminus \{0\}} \frac{\left\| \nabla u \right\|_p^p}{\left\| u\right\|_{p^*}^p}
		\end{equation}
		and is achieved by
		the family of functions of the form:
		\begin{equation}\label{U_epsilon}
			U_{\epsilon,x_0}(x)=\frac{K_{N,p}\epsilon^{\frac{N-p}{p(p-1)}}}{(\epsilon^{\frac{p}{p-1}}+|x-x_0|^{\frac{p}{p-1}})^{\frac{N-p}{p}}},\text{ for } x_0\in \mathbb{R}^N \text{ and } \epsilon>0.
		\end{equation}
		\item $p_{\alpha}^*:=\frac{p}{2}\left(\frac{N+\alpha}{N-p}\right)$.
	\end{itemize}
			\section{Regularity Result}
	\noindent Let us start with dicussing the regularity properties of a positive radial solution to \eqref{1.1}-\eqref{1.2}. 
	\begin{myproof}{Theorem}{\ref{Theorem 1.1}}
		For $0<\epsilon<1$, define $h_{\epsilon}(t):= \sqrt{\epsilon^2+t^2}$, clearly, $g_{\epsilon}:=h'_{\epsilon}$ is a continuously differentiable convex function, with $g_{\epsilon}(0)=0$ and $|g'_{\epsilon}(t)|\leq M_{\epsilon}$ for some $M_{\epsilon}>0$, thus by Theorem 2.2.3 of  \cite{Kesavan2019Topics}, $g_{\epsilon}(u)\in W^{1,p}(\mathbb{R}^N)$. Let $0<\phi\in C_c^{\infty}(\mathbb{R}^N)$ be arbitrary, now define 
		$\zeta := \phi |g_{\epsilon}(u)|^{p-2}g_{\epsilon}(u)\in W^{1,p}(\mathbb{R}^N)$, then we have: 
		\begin{equation}\label{2.1}
			\int_{\mathbb{R}^N}|\nabla u|^{p-2}\nabla u \nabla \zeta =	(p-1)\int_{\mathbb{R}^N}|\nabla u|^p\phi|g_{\epsilon}(u)|^{p-2}g'_{\epsilon}+\int_{\mathbb{R}^N}|\nabla u|^{p-2}|g_{\epsilon}(u)|^{p-2}g_{\epsilon}(u)\nabla u\nabla \phi:= I_{1}^{\epsilon},
		\end{equation}
		\begin{eqnarray}\label{2.2}
			\ll u, \zeta \gg_{s,p} & = & \int_{\mathbb{R}^N}\int_{\mathbb{R}^N}\frac{|u(x)-u(y)|^{p-2}(u(x)-u(y))(\zeta(x)-\zeta(y))}{|x-y|^{N+sp}}dxdy\nonumber\\
			& \geq & \int_{\mathbb{R}^N} \int_{\mathbb{R}^N}\frac{|h_{\epsilon}(u(x))-h_{\epsilon}(u(y))|^{p-2}(h_{\epsilon}(u(x))-h_{\epsilon}(u(y)))(\phi(x)-\phi(y))}{|x-y|^{N+sp}}\nonumber\\
			& := & I_2^{\epsilon}
		\end{eqnarray}
		by Lemma 3.2 of \cite{Biswas2023Regularity}. Taking $\zeta$ as test function, by \eqref{2.1} and \eqref{2.2} we get:
		\begin{eqnarray}\label{2.3}
			I_1^{\epsilon}+I_2^{\epsilon}+(-\lambda)\int_{\mathbb{R}^N}|u|^{p-2}u \zeta & \leq & \int_{\mathbb{R}^N}|\nabla u|^{p-2} \nabla u \nabla \zeta + \ll u,\zeta \gg_{s,p}-\lambda\int_{\mathbb{R}^N}|u|^{p-2}u \zeta\nonumber\\
			& = & \int_{\mathbb{R}^N}|u|^{p^*-2}u\zeta +\mu\int_{\mathbb{R}^N}(I_{\alpha}*|u|^q)|u|^{q-2}u\zeta\nonumber\\
			& \leq & \int_{\mathbb{R}^N}|u|^{p^*-2}u\phi +\mu\int_{\mathbb{R}^N}(I_{\alpha}*|u|^q)|u|^{q-2}u \phi,
		\end{eqnarray}
		since $|g_{\epsilon}(t)|<1$. Now, taking $\epsilon\rightarrow 0$ in \eqref{2.3}, by dominated convergence theorem, we get :
		\begin{equation}\label{2.4}
			\int_{\mathbb{R}^N}|\nabla u|^{p-2}\nabla u \nabla \phi +\ll |u|, \phi \gg_{s,p} \leq \lambda \int_{\mathbb{R}^N}|u|^{p-2}u \phi+\int_{\mathbb{R}^N}|u|^{p^*-2}u\phi +\mu\int_{\mathbb{R}^N}(I_{\alpha}*|u|^q)|u|^{q-2}u \phi
		\end{equation}
		for all $0<\phi \in C_c^{\infty}(\mathbb{R}^N)$ and hence by density, for all $0<\phi \in W^{1,p}(\mathbb{R}^N)$. Now, for a fixed $\gamma>0$, define $u_{\gamma}:=\min\{ \gamma, |u|\}$ and for some $k > 1$, set $\beta=kp-p+1$. Replacing $\phi$ by $u_{\gamma}^{\beta}$ in \eqref{2.4}, then by Lemma 3.1 of \cite{Biswas2023Regularity} we get:
		\begin{eqnarray*}
			\frac{\beta p^p}{(\beta+p-1)^p}\left(\left\| \nabla(u_{\gamma}^{k})\right\|_p^p+[u_{\gamma}^k]_{s,p}^p\right) & \leq & \beta \int_{\{x: |u(x)|<\gamma\}}u_{\gamma}^{\beta-1}|\nabla u_{\gamma}|^p+ \ll |u|, u_{\gamma}^{\beta} \gg_{s,p}\\
			& = & \int_{\mathbb{R}^N}|\nabla u|^{p-2}\nabla u \nabla u_{\gamma}^{\beta} +\ll |u|, u_{\gamma}^{\beta} \gg_{s,p}\\
			& \leq & \lambda \int_{\mathbb{R}^N}|u|^{p-2}u u_{\gamma}^{\beta}+\int_{\mathbb{R}^N}|u|^{p^*-2}u u_{\gamma}^{\beta}\\
			&& +\mu\int_{\mathbb{R}^N}(I_{\alpha}*|u|^q)|u|^{q-2}u u_{\gamma}^{\beta}.
		\end{eqnarray*}
		Now, since $u>0$ and $\lambda<0$ we get
		$$\frac{\beta}{k^p}\left(\left\| \nabla u_{\gamma}^k\right\|_p^p+[u_{\gamma}^k]_{s,p}^p\right)+(-\lambda)\left\| u_{\gamma}^k\right\|_p^p\leq \int_{\mathbb{R}^N}|u|^{p^*-2}u u_{\gamma}^{\beta}+\mu\int_{\mathbb{R}^N}(I_{\alpha}*|u|^q)|u|^{q-2}u u_{\gamma}^{\beta},$$
		and hence 
		$$\frac{C\beta}{k^p}\left\| u_{\gamma}^k\right\|^p\leq \int_{\mathbb{R}^N}|u|^{p^*-2}u u_{\gamma}^{\beta}+\mu\int_{\mathbb{R}^N}(I_{\alpha}*|u|^q)|u|^{q-2}u u_{\gamma}^{\beta},$$
		here $C=\min\{1, \frac{(-\lambda)k^p}{\beta}\}$. Thus, by the imbedding $W^{1,p}(\mathbb{R}^N)\hookrightarrow L^{p^*}(\mathbb{R}^N)$, there exists a constant $C_1$ such that
		\begin{equation}\label{2.5}
			\left\| u_{\gamma}^k\right\|_{p^*}^{p}\leq C_1 k^{p-1}\left(\int_{\mathbb{R}^N}|u|^{p^*-2}uu_{\gamma}^{\beta}+\mu\int_{\mathbb{R}^N}(I_{\alpha}*|u|^q)|u|^{q-2}u u_{\gamma}^{\beta}\right),
		\end{equation}
		because, $\beta>k$ for $k>1$. Now, for some $\delta>1$, by \autoref{prop1.1}, H$\ddot{\text{o}}$lder's inequality, and the fact that $(a+b)^c\leq a^c+b^c$ whenever $a$ , $b>0$ and $c<1$, we have:
		\begin{eqnarray}\label{2.6}
			&&\mu\int_{\mathbb{R}^N}(I_{\alpha}*|u|^q)|u|^{q-2}u u_{\gamma}^{\beta} \nonumber\\
			&&\leq  \mu C_N\left(\int_{\mathbb{R}^N}|u|^{\frac{2qN}{N+\alpha}}\right)^{\frac{N+\alpha}{2N}}\left(\int_{\mathbb{R}^N}(|u|^{q-2}|uu_{\gamma}^{\beta}|)^{\frac{2N}{N+\alpha}}\right)^{\frac{N+\alpha}{2N}} \nonumber\\
			&& \leq  \bar{C}_N\left(\int_{\{x: |u(x)|<\delta\}}(|u(x)|^{q-2}|uu_{\gamma}^{\beta}|)^{\frac{2N}{N+\alpha}}+\int_{\{x |u(x)|\geq\delta\}}(|u|^{q-2}|uu_{\gamma}^{\beta}|)^{\frac{2N}{N+\alpha}}\right)^{\frac{N+\alpha}{2N}}\nonumber\\
			&& \leq  \bar{C}_N\left(\int_{\{x: |u(x)|<\delta\}}(|u|^{q-2}|uu_{\gamma}^{\beta}|)^{\frac{2N}{N+\alpha}}\right)^{\frac{N+\alpha}{2N}}+\bar{C}_N\left(\int_{\{x: |u(x)|\geq\delta\}}(|u|^{q-2}|uu_{\gamma}^{\beta}|)^{\frac{2N}{N+\alpha}}\right)^{\frac{N+\alpha}{2N}}\nonumber\\
			&& \leq \bar{C}_N\left(\int_{\{x: |u(x)|<\delta\}}(|u|^{q+kp-p})^{\frac{2N}{N+\alpha}}\right)^{\frac{N+\alpha}{2N}}+\bar{C}_N\left(\int_{\{x: |u(x)|\geq\delta\}}(|u|^{q+kp-p}|)^{\frac{2N}{N+\alpha}}\right)^{\frac{N+\alpha}{2N}}\nonumber\\
			&& \leq \bar{C}_N\delta^{q-p}\left(\int_{\{x: |u(x)|<\delta\}}|u|^{\frac{kpp^*}{p^*_{\alpha}}}\right)^{\frac{p^*_{\alpha}}{p^*}}+\bar{C}_N\left(\int_{\{x |u(x)|\geq\delta\}}(|u|^{p^*_{\alpha}-p}|u|^{kp})^{\frac{p^*}{p^*_{\alpha}}}\right)^{\frac{p^*_{\alpha}}{p^*}}\nonumber\\
			&& \leq \bar{C}_N \delta^{q-p}\left\| u \right\|_{\frac{kpp^*}{p^*_{\alpha}}}^{kp} +\bar{C}_N\left(\int_{\{x |u(x)|\geq\delta\}}|u|^{kp^*}\right)^{\frac{p}{p^*}}\left(\int_{\{x |u(x)|\geq\delta\}}|u|^{p^*}\right)^{\frac{p^*_{\alpha}-p}{p^*}}\nonumber\\
			&& \leq \bar{C}_N\delta^{q-p}\left\| u \right\|_{\frac{kpp^*}{p^*_{\alpha}}}^{kp}+\bar{C}_NC(\delta)^{\frac{p^*_{\alpha}-p}{p^*}}\left\| u \right\|_{kp^*}^{kp},
		\end{eqnarray} 
		and
		\begin{eqnarray}\label{2.7}
			\int_{\mathbb{R}^N}|u|^{p^*-2}uu_{\gamma}^{\beta} & \leq & \int_{\mathbb{R}^N}|u|^{p^*+\beta-1}= \int_{\mathbb{R}^N}|u|^{p^*+kp-p}\nonumber\\
			& = & \int_{\{x: |u(x)|<\delta\}}|u|^{p^*+kp-p}+\int_{\{x: |u(x)|\geq\delta\}}|u|^{p^*+kp-p}\nonumber\\
			& \leq &  \left(\int_{\{x: |u(x)|<\delta\}}|u|^{\frac{kpp^*}{p^*_{\alpha}}}\right)^{\frac{p^*_{\alpha}}{p^*}}\left(\int_{\{x: |u(x)|<\delta\}}|u|^{\frac{(p^*-p)p^*}{p^*-p^*_{\alpha}}}\right)^{\frac{p^*-p^*_{\alpha}}{p^*}}\nonumber\\
			& & + \left(\int_{\{x |u(x)|\geq\delta\}}|u|^{kp^*}\right)^{\frac{p}{p^*}}\left(\int_{\{x |u(x)|\geq\delta\}}|u|^{p^*}\right)^{\frac{p^*-p}{p^*}}\nonumber\\
			& \leq &  \left\| u \right\|_{\frac{kpp^*}{p^*_{\alpha}}}^{kp}\left(\int_{\{x: |u(x)|<\delta\}}|u|^{\frac{(p^*-p^*_{\alpha}+p^*_{\alpha}-p)p^*}{p^*-p^*_{\alpha}}}\right)^{\frac{p^*-p^*_{\alpha}}{p^*}}+C(\delta)^{\frac{p^*-p}{p^*}}\left\| u \right\|_{kp^*}^{kp}\nonumber\\
			& \leq & C'\delta^{p^*_{\alpha}-p}\left\| u \right\|_{\frac{kpp^*}{p^*_{\alpha}}}^{kp}+C(\delta)^{\frac{p^*-p}{p^*}}\left\| u \right\|_{kp^*}^{kp}.
		\end{eqnarray}
		Using \eqref{2.6} and \eqref{2.7} in \eqref{2.5} we get :
		\begin{eqnarray*}
			\left\| u_{\gamma}^k \right\|_{p^*}^{p} & \leq & C_1k^{p-1}\left(\bar{C}_N\delta^{q-p}+C'\delta^{p^*_{\alpha}-p}\right)\left\| u \right\|_{\frac{kpp^*}{p^*_{\alpha}}}^{kp}\\
			&& + C_1 k^{p-1}\left(\bar{C}_NC(\delta)^{\frac{p^*_{\alpha}-p}{p^*}}+C(\delta)^{\frac{p^*-p}{p^*}}\right)\left\| u \right\|_{kp^*}^{kp},
		\end{eqnarray*}
		taking $\delta>1$ large enough so that $ C'':=C_1 k^{p-1}\left(\bar{C}_NC(\delta)^{\frac{p^*_{\alpha}-p}{p^*}}+C(\delta)^{\frac{p^*-p}{p^*}}\right)<1$. Thus, Fatou's lemma gives us the following :
		\begin{eqnarray*}
			\left\| u \right\|_{kp^*}^{kp} & = & \left\| u^k \right\|_{p^*}^p \leq \liminf_{\gamma \rightarrow \infty} \left\| u_{\gamma}^k \right\|_{p^*}^{p}\\
			& \leq &C_1k^{p-1}\left(\bar{C}_N\delta^{q-p}+C'\delta^{p^*_{\alpha}-p}\right)\left\| u \right\|_{\frac{kpp^*}{p^*_{\alpha}}}^{kp}+ C''\left\| u \right\|_{kp^*}^{kp},
		\end{eqnarray*} 
		therefore,
		\begin{equation}\label{2.8}
			\left\| u \right\|_{kp^*} \leq \hat{C}^{\frac{1}{kp}}k^{\frac{p-1}{kp}}\left\| u \right\|_{\frac{kpp^*}{p^*_{\alpha}}} \leq k^{\frac{p-1}{kp}}\hat{C}^{\frac{1}{kp}}\left(1+\left\| u \right\|_{\frac{kpp^*}{p^*_{\alpha}}}^{kp}\right)^{\frac{1}{kp}},
		\end{equation}
		where $\hat{C}=\frac{C_1k^{p-1}\left(\bar{C}_N\delta^{q-p}+C'\delta^{p^*_{\alpha}-p}\right)}{1-C''}>1$. If there exists a sequence $\{k_n\}\rightarrow \infty$ as $n\rightarrow \infty$ such that $\left\| u \right\|_{\frac{k_npp^*}{p^*_{\alpha}}}\leq 1$ for all $n\in \mathbb{N}$. Then, for any bounded domain $\Omega\subset\mathbb{R}^N$, we have
		$$\left\| u \right\|_{L^{\frac{k_npp^*}{p^*_{\alpha}}}(\Omega)}\leq \left\| u \right\|_{\frac{k_npp^*}{p^*_{\alpha}}}\leq 1 \text{ for all } n\in \mathbb{N},$$
		taking $n\rightarrow \infty$, we get $u\in L^{\infty}(\Omega)$ and hence $u\in L^{\infty}_{loc}(\mathbb{R}^N)$. Now, suppose there does not exists any such sequence, then there exists $k_0>0$ such that
		$$\left\| u \right\|_{\frac{kpp^*}{p^*_{\alpha}}}^{\frac{kpp^*}{p^*_{\alpha}}}=\int_{\mathbb{R}^N}|u|^{\frac{kpp^*}{p^*_{\alpha}}}>1 \text{ for all } k\geq k_0,$$
		and hence by \eqref{2.8} we get 
		\begin{equation}\label{2.9}
			\left\|u \right\|_{kp^*}\leq (2\hat{C})^{\frac{1}{kp}}(k^{\frac{1}{k}})^{\frac{p-1}{p}}\left\| u \right\|_{\frac{kpp^*}{p^*_{\alpha}}} = (C_*^{\frac{1}{k}})^{\frac{1}{p}}(k^{\frac{1}{k}})^{\frac{p-1}{p}}\left\| u \right\|_{\frac{kpp^*}{p^*_{\alpha}}}\text{ for all } k\geq k_0,
		\end{equation}
		where $C_*=2\hat{C}>1$. For $k_1=\frac{k_0p^*_{\alpha}}{p}>k_0$, \eqref{2.9} gives us
		\begin{equation*}
			\left\| u \right\|_{k_1p^*} \leq (C_*^{\frac{1}{k_1}})^{\frac{1}{p}}(k_1^{\frac{1}{k_1}})^{\frac{p-1}{p}}\left\| u \right\|_{k_0p^*},
		\end{equation*}
		similarly for $k_2=\frac{k_1p^*_{\alpha}}{p}$ we have:
		$$\left\| u \right\|_{k_2p^*}\leq (c_*^{\frac{1}{k_1}+\frac{1}{k_2}})^{\frac{1}{p}}(k_1^{\frac{1}{k_1}}k_2^{\frac{1}{k_2}})^{\frac{p-1}{p}}\left\| u \right\|_{k_0p^*},$$
		proceeding in this manner, for $k_n=\frac{k_{n-1}p^*_{\alpha}}{p}=(\frac{p^*_{\alpha}}{p})^nk_0$ we have:
		\begin{equation}\label{2.10}
			\left\| u \right\|_{k_np^*}\leq \left(C_*^{\sum_{j=1}^{n}\frac{1}{k_j}}\right)^{\frac{1}{p}}\left(\prod_{j=1}^nk_j^{\frac{1}{k_j}}\right)^{\frac{p-1}{p}}\left\| u \right\|_{k_0p^*}.
		\end{equation}
		Now, since $k_j>1$ for all $j\in \mathbb{N}$, and $\displaystyle \lim_{j \rightarrow \infty}k_j^{\sqrt{\frac{1}{k_j}}}=1$, there must exist some $C^*>1$ such that $k_j^{\sqrt{\frac{1}{k_j}}}<c^*$ for all $j\in \mathbb{N}$, thus \eqref{2.10}, we get:
		\begin{eqnarray*}
			\left\| u \right\|_{k_np^*} & \leq &  \left(C_*^{\sum_{j=1}^{n}\frac{1}{k_j}}\right)^{\frac{1}{p}}\left((C^*)^{\sum_{j=1}^{n}\frac{1}{\sqrt{k_j}}}\right)^{\frac{p-1}{p}}\left\| u \right\|_{k_op^*}\\
			& \leq & \left(C_*^{\frac{k_0p}{p^*_{\alpha}-p}}\right)^{\frac{1}{p}}\left((C^*)^{\frac{\sqrt{k_0p}}{\sqrt{p^*_{\alpha}}-\sqrt{p}}}\right)^{\frac{p-1}{p}}\left\| u \right\|_{k_0p^*}=\tilde{C} \left\| u \right\|_{k_0p^*},
		\end{eqnarray*}		 		
		thus, for any bounded domain $\Omega$ in $\mathbb{R}^N$, we have:
		\begin{equation}\label{2.11}
			\left\| u \right\|_{L^{k_np^*}(\Omega)}\leq \left\| u \right\|_{k_np^*}\leq \tilde{C} \left\| u \right\|_{k_0p^*} \text{ for all } n\in \mathbb{N}.
		\end{equation}
		Now, since $k_n\rightarrow \infty$ as $n\rightarrow \infty$, then we must have $u\in L^{\infty}(\Omega)$, because, if not so, then there must exists a set $S\subset \Omega$ with positive measure and $\omega>0$, such that
		$$u(x)>\tilde{C}\left\| u \right\|_{k_0p^*}+\omega \text{ for all } x\in S,$$
		and hence,
		\begin{eqnarray*}
			\liminf_{n \rightarrow \infty} \left\| u \right\|_{L^{k_np^*}(\Omega)} & = & \liminf_{n \rightarrow \infty}\left(\int_{\Omega}|u|^{k_np^*}\right)^{\frac{1}{k_np^*}} \geq \liminf_{n \rightarrow \infty}\left(\int_{S}|u|^{k_np^*}\right)^{\frac{1}{k_np^*}}\\
			& > & \liminf_{n \rightarrow \infty}\left(\int_{S}|\tilde{C}\left\|u\right\|_{k_0p}+\omega|^{k_np^*}\right)^{\frac{1}{k_np^*}} = (\tilde{C} \left\| u \right\|_{k_0p^*}+\omega)|S|^{\frac{1}{k_np^*}}\\
			& = &\tilde{C} \left\| u \right\|_{k_0p^*}+\omega,
		\end{eqnarray*}
		but this is a contradiction to \eqref{2.11}. Thus, $u\in L^{\infty}_{loc}(\mathbb{R}^N)$. Now, if $u$ is radial, then by Strauss lemma \cite{Yuan2013radial} we have
		\begin{equation}\label{radial_estimate}
			|u(x)|\leq \frac{K \left\| u\right\|_{W^{1,p}(\mathbb{R}^N)}}{|x|^{\frac{N-1}{p}}} \text{ almost everywhere in }\mathbb{R}^N.
		\end{equation}
		The estimate in \eqref{radial_estimate} tells us that $|u(x)|\rightarrow 0$ as $|x|\rightarrow \infty$ and hence $u\in L^{\infty}(\mathbb{R}^N)$. Moreover, using \eqref{radial_estimate} we can see that, $u\in L^t_{loc}(\mathbb{R}^N)$ for all $t\geq 1$ and $u\in L^t(\mathbb{R}^N)$ for all $t\geq p$. Now, as done in \cite{Nidhi2023}, one can see that $(I_{\alpha}*|u|^q)\in L^{\infty}(\mathbb{R}^N)$ and hence $$f:= \lambda |u|^{p-2}u+|u|^{p^*-2}u+(I_{\alpha}*|u|^q)|u|^{q-2}u\in L^t(\mathbb{R}^N) \text{ for all } t\geq \max\{\frac{p}{p-1},\frac{p}{q-1}\}.$$
		Thus, by theorem 1.4 of \cite{Garain2023Higher} $u\in C^{\delta}_{loc}(\mathbb{R}^N)$ for every $0<\delta<\Theta=\min\{\frac{sp}{p-1},1\}$.
	\end{myproof}\\
	Following the establishment of the regularity of a radial solution, we shall examine its existence. To proceed, we require the following preliminary results:
	\section{Preliminaries for existence results}
	\begin{proposition}\label{prop 3.1}
		For all $u\in W^{1,p}(\mathbb{R}^N)$, we have:
		\begin{equation}\label{G_N}
			A(u)\leq C_N\left\| \nabla u \right\|_{p}^{2q\gamma_q}\left\| u\right\|_p^{2q(1-\gamma_q)},
		\end{equation}
		and
		\begin{equation}\label{fractionalGN}
			A(u) \leq C_N [u]_{s,p}^{\frac{2q\gamma_q}{s}}\left\| u\right\|_p^{\frac{2q(s-\gamma_q)}{s}}.
		\end{equation}
	\end{proposition}
	\begin{proof}
		For $r,t>1$ such that $\frac{1}{r}+\frac{1}{t}=1+\frac{\alpha}{N}$, by \autoref{prop1.1} and the interpolation inequality \cite{Fiorenza2021Gagliardo}, we get:
		\begin{eqnarray*}
			A(u) & \leq & \mu C(N,r,t)\left\| u^q\right\|_r\left\| u^q\right\|_t\\
			& \leq & \mu\bar{C}(N,r,t)\left(\left\| \nabla u \right\|_p^{\frac{N(qr-p)}{p}}\left\| u \right\|_p^{qr-\frac{N(qr-p)}{p}}\right)^{\frac{1}{r}}\left(\left\| \nabla u \right\|_p^{\frac{N(qt-p)}{p}}\left\| u \right\|_p^{qt-\frac{N(qt-p)}{p}}\right)^{\frac{1}{t}}\\
			& = & C_N\left\|\nabla u \right\|_p^{2q\gamma_q}\left\| u\right\|_p^{2q(1-\gamma_q)},
		\end{eqnarray*}
		where $C_N:=\mu \bar{C}(N,r,t)$.
		Similarly, by the fractional Gagliardo-Nirenberg inequality \cite{Park2011FractionalGN}, 
		\begin{eqnarray*}
			A(u) & \leq & \mu C(N,r,t)\left\| u^q\right\|_r\left\| u^q\right\|_t\\
			& \leq & \mu\bar{C}(N,r,t)\left([u]_{s,p}^{\frac{N(qr-p)}{ps}}\left\| u \right\|_p^{qr-\frac{N(qr-p)}{ps}}\right)^{\frac{1}{r}}\left([u]_{s,p}^{\frac{N(qt-p)}{ps}}\left\| u \right\|_p^{qr-\frac{N(qt-p)}{ps}}\right)^{\frac{1}{t}}\\
			& = & C_N [u]_{s,p}^{\frac{2q\gamma_q}{s}}\left\| u\right\|_p^{\frac{2q(s-\gamma_q)}{s}},
		\end{eqnarray*}
		where $C_N:=\mu \bar{C}(N,r,t)$.
	\end{proof}
	\begin{lemma}\label{Lemma 3.1}
		Let $N\geq 3, (u,\lambda)\in S(\tau)\times \mathbb{R}$ be a weak solution of \eqref{1.1}-\eqref{1.2}, then $u\in P(\tau)$.
	\end{lemma}
	\begin{proof}
		By, \cite{Avenia2015} and \cite{Moroz2015} we can see that $u$ satisfies the Pohozaev identity :
		\begin{equation}\label{3.3}
			\left(\frac{N-p}{p}\right)\left\| \nabla u\right\|_p^p+\left(\frac{N-sp}{p}\right)[u]_{s,p}^p=\frac{N\lambda}{p}\left\|u \right\|_p^p+\frac{N}{p^*}\left\| u\right\|_{p^*}^{p^*}+\left(\frac{N+\alpha}{2q}\right)A(u),
		\end{equation}
		and since $u$ is a solution of \eqref{1.1}, we have:
		\begin{equation}\label{3.4}
			\left\| u\right\|_p^p+[u]_{s,p}^p-\lambda\left\| u\right\|_{p}^p-\left\| u\right\|_{p^*}^{p^*}-A(u)=0.
		\end{equation}
		By \eqref{3.3} and \eqref{3.4} we get
		\begin{equation*}
			M(u)= \left\| \nabla u \right\|_p^p+s[u]_{s,p}^p-\left\| u\right\|_{p^*}^{p^*}-\gamma_qA(u)=0.
		\end{equation*}
	\end{proof}
	\noindent For $u\in W^{1,p}(\mathbb{R}^N)$ define $$I_u(t):=E(t\star u)=\frac{e^{pt}}{p}\left\| \nabla u \right\|_p^p+\frac{e^{spt}}{p}[u]_{s,p}^p-\frac{e^{p^*t}}{p^*}\left\| u\right\|_{p^*}^{p^*}-\frac{e^{2q\gamma_q}}{2q}A(u).$$
	Let us examine the critical points of $I_u$.
	\begin{lemma}\label{Lemma 3.2}
		For $u\in S(\tau)$, $t\in \mathbb{R}$ is a critical point of $I_u$ if and only if $t\star u \in P(\tau)$.
	\end{lemma}
	\begin{proof}
		For any $u\in S(\tau)$ we have:
		\begin{equation*}
			I_u'(t)= e^{pt}\left\| \nabla u \right\|_p^p+se^{spt}[u]_{s,p}^p-e^{p^*t}\left\| u\right\|_{p^*}^{p^*}-\gamma_qe^{2q\gamma_qt}A(u)=M(t\star u),
		\end{equation*}
		and $\left\| t\star u\right\|_p^p=\left\| u\right\|_p^p=\tau$. Therefore, $t\in \mathbb{R}$ is a critical point of $I_u$ if and only if $t\star u\in P(\tau)$.
	\end{proof}
	\section{Case  \texorpdfstring{$1$}{TEXT} : \texorpdfstring{$\frac{p}{2}\left(\frac{N+\alpha}{N}\right)<q<\bar{q_s}:=\frac{p}{2}\left(\frac{N+\alpha+sp}{N}\right)$}
		{TEXT} }
	By \eqref{S} and \autoref{prop 3.1}, for $u\in S(\tau)$, we have:
	\begin{eqnarray}\label{E(u)}
		E(u) & \geq & \frac{1}{p}T(u)-\frac{1}{p^*S^{\frac{p^*}{p}}}T(u)^{\frac{p^*}{p}}-\frac{C_N}{2q}T(u)^{\frac{2q\gamma_q}{p}}\tau^{\frac{2q(1-\gamma_q)}{p}}\nonumber\\
		& = & T(u)\left(\frac{1}{p}-\frac{1}{p^*S^{\frac{p^*}{p}}}T(u)^{\frac{p^*}{N}}-\frac{C_N}{2q}T(u)^{\frac{2q\gamma_q-p}{p}}\tau^{\frac{2q(1-\gamma_q)}{p}}\right) .
	\end{eqnarray}
	Let us define
	\begin{equation}\label{h}
		h_{\tau}(r):= \frac{1}{p}-\frac{1}{p^*S^{\frac{p^*}{p}}}r^{\frac{p^*}{N}}-\frac{C_N}{2q}r^{\frac{2q\gamma_q-p}{p}}\tau^{\frac{2q(1-\gamma_q)}{p}} \text{ for all } r>0.
	\end{equation}
	and look for a subset $\Lambda(\tau)$ of $S(\tau)$, such that $\displaystyle \min_{u\in P(\tau)}E(u)=\min_{u\in \Lambda(\tau)}E(u)$, using the following property of $h_{\tau}$.
	\begin{lemma}\label{Lemma 4.1}
		There exists  $\tau_0>0$ such that, for every $\tau\in (0,\tau_0)$, $h_{\tau}(r)$ has unique global maximum at some $r_{\tau}>0$, and  there exist $r_0>0$ such that $h_{\tau}(r)\geq 0$ for all $r\in [\frac{\tau}{\tau_0}r_0,r_0]$.
	\end{lemma}
	\begin{proof}
		Since, $2q\gamma_q-p<0$, we can see that $h_{\tau}(r)\rightarrow -\infty$ as $r\rightarrow +\infty$, $h_{\tau}(r)\rightarrow -\infty$ as $r\rightarrow 0$, and
		\begin{equation*}
			h_{\tau}'(r) = -\frac{1}{NS^{\frac{p^*}{p}}}r^{\frac{p^*}{N}-1}-\frac{C_N(2q\gamma_q-p)\tau^{\frac{2q(1-\gamma_q)}{p}}}{2qp}r^{\frac{2q\gamma_q}{p}-2},
		\end{equation*}
		then clearly, \begin{eqnarray}\label{r_tau}
			r_{\tau}= \left(\frac{NC_N(p-2q\gamma_q)S^{\frac{p^*}{p}}}{2pq}\right)^{\frac{p}{p^*-2q\gamma_q}}\tau^{\frac{2q(1-\gamma_q)}{p^*-2q\gamma_q}}>0
		\end{eqnarray}
		is the global maxima for  $h_{\tau}$ with
		\begin{equation*}
			\max_{r\in \mathbb{R}}h_{\tau}(r) = h_{\tau}(r_{\tau})=\frac{1}{p}-K\tau^{\frac{2q(1-\gamma_q)(p^*-p)}{p(p^*-2q\gamma_q)}}
		\end{equation*}
		where $K=\frac{C_N}{2q}(\frac{NC_N(p-2q\gamma_q)S^{\frac{p^*}{p}}}{2pq})(1+\frac{N(p-2q\gamma_q)S^{\frac{p^*}{p}}}{pp^*})>0$. Let $\tau_0=(pK)^{\frac{N(2q\gamma_q-p^*)}{2q(1-\gamma_q)p^*}}$ and $r_0=r_{\tau_0}$, then 
		$$\max_{r\in \mathbb{R}}h_{\tau_0}(r) = h_{\tau_0}(r_{0})=\frac{1}{p}-K\tau_0^{\frac{2q(1-\gamma_q)(p^*-p)}{p(p^*-2q\gamma_q)}}=0,$$ 
		and hence, for every $\tau<\tau_0$
		\begin{equation}\label{h_tau(r_tau)}
			\max_{r\in \mathbb{R}}h_{\tau}(r)=h_{\tau}(r_{\tau})=\frac{1}{p}-K\tau^{\frac{2q(1-\gamma_q)(p^*-p)}{p(p^*-2q\gamma_q)}}>\frac{1}{p}-K\tau_0^{\frac{2q(1-\gamma_q)(p^*-p)}{p(p^*-2q\gamma_q)}}=h_{\tau_0}(r_0)=0.
		\end{equation}
		Now, for every $r$, define $\phi_r:(0,\tau_0)\rightarrow \mathbb{R}$ as
		$\phi_r(\tau):=h_{\tau}(r)$. Clearly, $\phi_r$ is  non-increasing for every $r$, and hence $h_{\tau}(r_0)> h_{\tau_0}(r_0)=0$ for each $\tau\in (0,\tau_0)$. 
		Also, since $2q-p>0$, then for every $\tau<\tau_0$ we have:
		\begin{eqnarray*}
			h_{\tau}(\frac{\tau}{\tau_0}r_0) & = & \frac{1}{p}-\frac{(\tau r_0)^{\frac{p^*}{N}}}{p^*\tau_0^{\frac{p^*}{N}}S^{\frac{p^*}{p}}}-\frac{C_N\tau^{\frac{2q-p}{p}}r_0^{\frac{2q\gamma_q-p}{p}}}{2q\tau_0^{\frac{2q\gamma_q-p}{p}}}\\
			& \geq & \frac{1}{p}-\frac{r_0^{\frac{p^*}{p}}}{p^*S^{\frac{p^*}{p}}}-\frac{C_N\tau_0^{\frac{2q(1-\gamma_q)}{p}}r_0^{\frac{2q\gamma_q-p}{p}}}{2q} = h_{\tau_0}(r_0)=0.
		\end{eqnarray*}
		Therefore, $h_{\tau}(r)\geq 0$ for all $r\in [\frac{\tau}{\tau_0}r_0,r_0]$.
	\end{proof}
	\noindent The above lemma motivates us to define: 
	$$\Lambda (\tau):=\{ u\in S(\tau): T(u)<r_0\},$$
	$$\partial \Lambda (\tau):=\{ u\in S(\tau): T(u)=r_0\},$$
	so, that we can study $m(\tau):=\displaystyle \min_{u\in \Lambda(\tau)}E(u).$
	Considering the notations:
	$$P(\tau)^+:=\{ u\in P(\tau) : E(u)<0\},$$
	$$P(\tau)^-:=\{u \in P(\tau): E(u)>0\}, $$
	we move forward to discuss some properties of $m(\tau)$ that will help us deduce the existence result. 
	\begin{lemma}\label{Lemma 4.2}
		For every $\tau \in (0,\tau_0)$,
		\begin{equation}\label{4.5}
			m(\tau)<0<\min_{u\in \partial\Lambda(\tau)}E(u).
		\end{equation}
	\end{lemma}
	\begin{proof}
		For $u\in \partial \Lambda(\tau)$, by \eqref{E(u)} and \eqref{h_tau(r_tau)}, we have:
		$$E(u)\geq r_0h_{\tau}(r_0)>r_0h_{\tau_0}(r_0)=0.$$
		Now, since $2q\gamma_q<sp<p<p^*$, we get  
		$$E(t\star u )= \frac{e^{pt}}{p}\left\| \nabla u \right\|_p^p+\frac{e^{spt}}{p}[u]_{s,p}^p-\frac{e^{p^*t}}{p^*}\left\| u\right\|_{p^*}^{p^*}-\frac{e^{2q\gamma_qt}}{2q}A(u)\rightarrow 0^- \text{ as } t \rightarrow -\infty,$$
		and hence, there exists $t_0<-1$ such that $E(t_0\star u )<0$ and $T(t_0\star u)<r_0$.
		Therefore, we get \eqref{4.5}.
	\end{proof}
	\begin{lemma}\label{Lemma 4.3}
		$m(\tau)=\displaystyle \min_{u\in P(\tau)}E(u)$.
	\end{lemma}
	\begin{proof}
		For $u\in S(\tau)$, as done in previous lemma we can see that, $I_u(t)\rightarrow 0^-$, $T(t\star u)\rightarrow 0$ as $t\rightarrow -\infty$ and $I_u(t)\rightarrow -\infty$ as $t\rightarrow +\infty$. Also, $I_u(t)>0$ for all $t$ such that $t\star u \in \partial\Lambda(\tau)$. Thus, $I_u$ must have atleast two crtical points, $\underline{t}_u\leq 0\leq \bar{t}_u$ where $\underline{t}_u$ corresponds to a local minima with $I_u(\underline{t}_u)<0$ and $\bar{t}_u$ to a local maxima with $I_u(\bar{t}_u)>0$. Thus by \autoref{Lemma 3.2}, $\underline{t}_u\star u \in P(\tau)^-$ and $\bar{t}_u\star u \in P(\tau)^+$.\\
		Claim 1:  $I_u$ has exactly two roots. \\
		Let if possible, $I_u$ has more than two roots, then $g$ must attain $C_q$ at more than two points where
		$$g(t):=e^{(p-2q\gamma_q)t}\left\| \nabla u \right\|_p^p+se^{(sp-2q\gamma_q)t}[u]_{s,p}^p-e^{(p^*-2q\gamma_q)t}\left\| u\right\|_{p^*}^{p^*},$$
		and
		$C_q=\gamma_qA(u)$. That is, $g$ must have atleast two critical points. Moreover, since every critical point of $g$ is a root of :
		$$f(t):= (p^*-2q\gamma_q)e^{(p^*-sp)t}\left\| u \right\|_{p^*}^{p^*}- (p-2q\gamma_q)e^{(p-sp)t}\left\| \nabla u \right\|_p^p-s(sp-2q\gamma_q)[u]_{s,p}^p,$$
		thus, $f$ must have atleast two roots. Now, let us examine the function $f$.
		We have:
		$$f'(t)=(p^*-sp)(p^*-2q\gamma_q)e^{(p^*-sp)t}\left\| u \right\|_{p^*}^{p^*}-(p-sp)(p-2q\gamma_q)e^{(p-sp)t}\left\| \nabla u \right\|_p^p,$$
		and $t_0=\frac{1}{p^*-p}\ln|\frac{(p-sp)(p-2q\gamma_q)\left\| \nabla u \right\|_p^p}{(p^*-sp)(p^*-2q\gamma_q)\left\| u \right\|_{p^*}^{p^*}}|$ is the unique critical point and global minima of $f$. Since $p^*>p$, we can see that $f(t)\rightarrow 0^-$ as $t\rightarrow -\infty$, thus $f(t_0)<0$ and $f$ can have atmost one root at some $t>t_0$. This gives us a contradiction, hence $I_u$ has exactly two roots $\underline{t}_u$ and $\bar{t}_u$.\\
		Claim 2 : $\underline{t}_u\star u \in \Lambda (\tau)$.\\
		Let if possible $T(\underline{t}_u\star u)\geq r_0$. Now, since $T(t\star u)\rightarrow 0$ as $t\rightarrow -\infty$ and $t\mapsto T(t\star u)$ is non-decreasing, therefore, there exists some $t_0<\underline{t}_u$ such that $T(t_0\star u)=r_0$, that is $t_0\star u\in \partial\Lambda (\tau)$. Thus, we get $I_u(t_0)=E(t_0\star u)>0$. But, since $t_0<\underline{t}_u$ and $I_u$ is non increasing on the interval $(-\infty, \underline{t}_u]$ with $\displaystyle \sup_{t\in (-\infty, \underline{t}_u]}I_u(t)=\lim_{t\rightarrow -\infty}I_u(t)=0^-$, thus $I_u(t_0)<0$. Hence, by contradiction, $\underline{t}_u\star u \in \Lambda (\tau)$.\\
		Now, for $u\in P(\tau)^-$ and $t<0$, define $v(x):=e^{-\frac{Nt}{p}}u(e^{-t}x)$. Then $t\star v =u \in P(\tau)^-\subset P(\tau)$, hence $I_v'(t)=0$. Thus, proceeding as in claim 2, we can see that $u=t\star v \in \Lambda(\tau)$. Therefore, $P(\tau)^-\subset \Lambda(\tau)$, which gives us:
		\begin{equation*}
			m(\tau)=\min_{u\in \Lambda(\tau)}E(u)\leq \min_{u\in P(\tau)^-}E(u)=\min_{u\in P(\tau)}E(u).
		\end{equation*}
		Also, if $u$ is a minimizer of $E$ on $\Lambda(\tau)$, then $I_u(t)\geq I_u(0)$ for all $t\leq \underline{t}_u$, since $t\star u\in \Lambda(\tau)$ for all $t\leq\underline{t}_u$. Moreover, since $I_u$ is increasing in $[\underline{t}_u,\bar{t}_u]$, we get $I_u(\underline{t}_u)=I_u(0)$ and hence $\underline{t}_u=0$. Therefore, $u=0\star u=\underline{t}_u\star u \in P(\tau)^-$. Thus,
		\begin{equation*}
			m(\tau)=E(u)\geq \min_{u\in P(\tau)^-}E(u)=\min_{u\in P(\tau)}E(u).
		\end{equation*}
	\end{proof}
	\begin{lemma}\label{Lemma 4.4}
		For every, $\tau_1$, $\tau_2>0$ such that $\tau_1+\tau_2=\tau\in (0,\tau_0)$, we have:
		$$m(\tau)\leq m(\tau_1)+m(\tau_2).$$
		Further, if either $m(\tau_1)$ or $m(\tau_2)$ is achieved, then
		$m(\tau)< m(\tau_1)+m(\tau_2).$
	\end{lemma}
	\begin{proof}
		Without loss of generality, let $\tau_1\leq\tau_2$. By definition of $m(\tau)$ and \autoref{Lemma 4.2}, we can find $u\in \Lambda(\tau_1)$ such that
		\begin{equation}\label{4.6}
			E(u) < m(\tau_1)+\epsilon \text{ and } E(u)<0.
		\end{equation}
		Then, by \eqref{E(u)}, we get $0>E(u)>T(u)h_{\tau_1}(T(u))$, hence by \autoref{Lemma 4.1} and the fact that $u\in \Lambda(\tau_1)$, we get 
		$$T(u)< \frac{\tau_1}{\tau}r_0.$$
		For $\delta=\frac{\tau_2}{\tau_1}\geq1$, let us define $\bar{u}(x):=u(\delta^{-\frac{1}{N}}x)$, then 
		$$T(\bar{u})=\delta(\delta^{-\frac{p}{N}}\left\| \nabla \right\|_p^p+\delta^{-\frac{sp}{N}}[u]_{s,p}^p)<\delta T(u)<\frac{\tau_2}{\tau}r_0<r_0.$$
		Thus, by \eqref{4.6}:
		$m(\tau_2)\leq E(\bar{u})\leq \delta E(u) <\delta (m(\tau_1)+\epsilon),$
		since $\epsilon>0$ is arbitrary , we get, 
		\begin{equation}\label{4.7}
			m(\tau_2)\leq \frac{\tau_2}{\tau_1}m(\tau_1).
		\end{equation}
		Also, since $\tau_2\leq \tau_1+\tau_2$, we can follow the same procedure to deduce
		\begin{equation}\label{4.8}
			m(\tau)=m(\tau_1+\tau_2)\leq \frac{(\tau_1+\tau_2)}{\tau_2}m(\tau_2)
		\end{equation}
		and hence by \eqref{4.7}, $m(\tau)\leq m(\tau_1)+m(\tau_2)$. Further, if $m(\tau_1)$ is achieved, then we can take $\epsilon=0$ in \eqref{4.6} will give us the required result. Similarly if $m(\tau_2)$ is achieved then we will get a strict inequality while proving \eqref{4.8} and hence the required result.
	\end{proof}
	\begin{lemma}\label{Lemma 4.5}
		The function $\phi : (0,\tau_0)\rightarrow \mathbb{R}$, defined as:
		$$\phi(\tau)= m(\tau)$$
		is continuous.
	\end{lemma}
	\begin{proof}
		Let $\{\tau_n\}\rightarrow \tau$ in $(0,\tau_0)$. For every $n\in \mathbb{N}$, let $u_n\in \Lambda(\tau_n)$ be such that:
		\begin{equation}\label{4.9}
			E(u_n)<m(\tau_n)+\frac{1}{n} \text{ and } E(u_n)<0.
		\end{equation}
		Define $v_n:=(\frac{\tau}{\tau_n})^{\frac{1}{p}}u_n$, clearly, $\{v_n\}$ is a sequence in $S(\tau)$, and $T(v_n)=\frac{\tau}{\tau_n}T(u_n)$ for every $n\in \mathbb{N}$. Now, since $E(u_n)<0$, then by \eqref{E(u)} and \autoref{Lemma 4.1}, we get $T(u_n)<\frac{\tau_n}{\tau_0}r_0$, and hence
		$$T(v_n)=\frac{\tau}{\tau_n}T(u_n)<\frac{\tau}{\tau_0}<r.$$
		Therefore, $\{v_n\}$ is a sequence in $\Lambda(\tau)$. Then, by \eqref{4.9} we get
		\begin{eqnarray}\label{4.10}
			m(\tau) & \leq & E(v_n) =\frac{\tau}{p\tau_n}\left\| \nabla u_n \right\|_p^p+\frac{\tau}{p\tau_n}[u_n]_{s,p}^p-\left(\frac{\tau}{\tau_n}\right)^{\frac{p^*}{p}}\frac{\left\| u_n\right\|_{p^*}^{p^*}}{p^*}-\left(\frac{\tau}{\tau_n}\right)^{\frac{2q}{p}}\frac{A(u_n)}{2q}\nonumber\\
			& = & E(u_n)+\left(\frac{\tau}{\tau_n}-1\right)\frac{\left\| \nabla u_n\right\|}{p}+\left(\frac{\tau}{\tau_n}-1\right)\frac{[u_n]_{s,p}^p}{p}-\left(\left(\frac{\tau}{\tau_n}\right)^{\frac{p^*}{p}}-1\right)\frac{\left\| u_n \right\|_{p^*}^{p^*}}{p^*}\nonumber\\
			& & -\left(\left(\frac{\tau}{\tau_n}\right)^{\frac{2q}{p}}-1\right)\frac{A(u_n)}{2q} = E(u_n)+o_n(1)\nonumber\\
			&  < & m(\tau_n)+\frac{1}{n}+o_n(1),
		\end{eqnarray}
		hence $m(\tau)\leq \displaystyle \liminf m(\tau_n)$. Now, let $\{w_n\}\subset \Lambda (\tau)$ be a minimizing sequence for $m(\tau)$ with $E(w_n)<0$. Setting $\bar{w}_n=(\frac{\tau_n}{\tau})^{\frac{1}{p}}w_n$, by \eqref{E(u)} and \autoref{Lemma 4.1} we get $\bar{w}_n\in \Lambda(\tau_n)$. Thus, as done in  \eqref{4.10} we get:
		\begin{eqnarray*}
			m(\tau_n)  \leq  E(\bar{w}_n)= E(w_n)+o_n(1) =m(\tau)+o_n(1),
		\end{eqnarray*}
		hence $m(\tau)\geq \limsup m(\tau_n)$. Therefore, $\phi(\tau)=m(\tau)=\displaystyle \lim_{n\rightarrow \infty} m(\tau_n)=\lim_{n\rightarrow \infty} \phi(\tau_n)$ and since, $\tau \in (0,\tau_0)$ is arbitrary, $\phi$ is continuous.
	\end{proof}
	\begin{proposition}\label{prop 4.1}
		If  $\{u_n\}\subset \Lambda(\tau)$ is a minimizing sequence for $E$ on $\Lambda(\tau)$, then either of the following holds: 
		\begin{enumerate}
			\myitem{(1)}\label{1} $\displaystyle \limsup_{n\rightarrow \infty} \left(\sup_{z\in \mathbb{R}^N}\int_{B_1(z)}|u_n|^2dx\right)=0.$
			\myitem{(2)}\label{2} There exists $u\in \Lambda (\tau)$ and a sequence $\{y_n\}\subset \mathbb{R}^N$ such that $u_n(.-y_n)\rightarrow u$ in $W^{1,p}(\mathbb{R}^N)$ as $n\rightarrow \infty$ and $E(u)=m(\tau)$.
		\end{enumerate}
	\end{proposition}
	\begin{proof}
		Let $\{u_n\}\subset \Lambda(\tau)$ be such that $E(u_n)\rightarrow m(\tau)$ as $n\rightarrow \infty$. Suppose that \ref{1} does not hold, then proving \ref{2} will suffice. Now since $\left\| u_n \right\|_p^p=\tau>0$ for all $n\in \mathbb{N}$, then we can find a sequence $\{y_n\}$ in $\mathbb{R}^N$ such that
		\begin{equation}\label{4.11}
			0<\lim_{n\rightarrow \infty}\int_{B_1(0)}|u_n(x-y_n)|^pdx \leq \tau.
		\end{equation}
		Claim 1 : $\{u_n\}$ is bounded in $W^{1,p}(\mathbb{R}^N)$.\\
		Since $\{u_n\}\subset \Lambda(\tau)$, we have $\left\| u_n \right\|^p=T(u_n)+\left\| u_n \right\|_p^p < r_0+\tau$ for all $n\in \mathbb{N}$. Hence, $\{u_n\}$ is bounded in $W^{1,p}(\mathbb{R}^N)$.\\
		Now, let $u\in W^{1,p}(\mathbb{R}^N)$ be such that $\{u_n(.-y_n)\}\rightharpoonup u$, weakly in $W^{1,p}(\mathbb{R}^N)$, upto a subsequence. Then, we have the following:
		\begin{equation}\label{4.12}
			\left\{
			\begin{array}{cc}
				\{u_n(.-y_n)\}  \rightarrow u & \text{ in } L^r_{loc}(\mathbb{R}^N) \text{ for all } r\in [p,p^*)\\
				\{u_n(.-y_n)\}  \rightarrow u  & \text{ almost everywhere in } \mathbb{R}^N.
			\end{array}
			\right.
		\end{equation}
		By \eqref{4.11} and \eqref{4.12}, it is clear that $u\neq 0$ and hence $\left\| u \right\|_p^p>0$. Setting $v_n:= u_n(.-y_n)-u$, by Brezis Lieb lemma, we get
		\begin{equation}\label{4.13}
			\left\{
			\begin{array}{ccc}
				\left\| v_n+u \right\|_p^p & = & \left\| v_n \right\|_p^p+\left\| u \right\|_p^p+o_n(1)\\
				\left\| v_n+u \right\|_{p^*}^{p^*} & = & \left\| v_n \right\|_{p^*}^{p^*}+\left\| u \right\|_{p^*}^{p^*}+o_n(1).\\
			\end{array}
			\right.
		\end{equation}
		Definig $B:W^{1,p}(\mathbb{R}^N)\rightarrow \mathbb{R}$ such that $B(u):=\frac{\left\| u \right\|^p}{p}$, then by \eqref{4.13}, we get:
		\begin{equation}\label{4.14}
			T(v_n+u) = T(v_n)+T(u)+o_n(1),
		\end{equation}
		also, by Brezis-lieb lemma for Reisz potential \cite{Moroz2017Guide}, we have:
		\begin{equation}\label{4.15}
			A(v_n+u) = A(v_n)+A(u)+o_n(1).
		\end{equation}
		Hence, by \eqref{4.13}-\eqref{4.15} we get
		\begin{eqnarray}\label{4.16}
			E(u_n) & = & E(u_n(.-y_n))=\frac{T(v_n+u)}{p}-\frac{\left\| v_n+u \right\|_{p^*}^{p^*}}{p^*}-\frac{A(v_n+u)}{2q}\nonumber\\
			& = & E(v_n)+E(u)+o_n(1).
		\end{eqnarray}
		Claim 2 : $\displaystyle \limsup_{n\rightarrow \infty} \sup_{z\in \mathbb{R}^N} \int_{B_1(z)}|v_n|^pdx=0$.\\
		Let if possible
		$$\limsup_{n\rightarrow \infty} \sup_{z\in \mathbb{R}^N} \int_{B_1(z)}|v_n|^pdx\neq 0,$$
		then by boundedness of $\{v_n\}$, as done above, we can find a sequence $\{z_n\}\subset \mathbb{R}^N$ and $v\in W^{1,p}(\mathbb{R}^N)$ such that upto a subsequence
		\begin{equation*}
			\begin{array}{cc}
				\{v_n(.-z_n)\} \rightharpoonup v & \text{ weakly in } W^{1,p}(\mathbb{R}^N),\\
				\{v_n(.-z_n)\} \rightarrow v & \text{ in } L^r_{loc}(\mathbb{R}^N) \text{ for all } r \in [p,p^*),\\
				\{v_n(.-z_n)\} \rightarrow v & \text{ almost everywhere in } \mathbb{R}^N.
			\end{array}				
		\end{equation*}
		Moreover, since $v\neq 0$, defining $w_n(x):= v_n(x-z_n)-v(x)$ and following the same procedure as done in \eqref{4.13}-\eqref{4.16}:
		\begin{equation}\label{4.17}
			\left\{
			\begin{array}{ccc}
				\left\| v_n \right\|_p^p & = & 	\left\| v \right\|_p^p+ \left\| w_n \right\|_p^p+o_n(1)\\
				T(v_n) & = & T(v)+T(w_n)+o_n(1)\\
				E(v_n) & = & E(v)+E(w_n)+o_n(1).
			\end{array}
			\right.
		\end{equation}
		Let $d:= \left\| u \right\|_p^p$, $b:= \left\| v \right\|_p^p$ and $l:=\tau-d-b$. Since $\{u_n\}\subset \Lambda(\tau)$, then by  \eqref{4.14}, we have $T(u)\leq T(v_n+u)=T(u_n)<r_0$, thus $u \in \Lambda(d)$, similarly, $v\in \Lambda(b)$. Now, using \eqref{4.17} in \eqref{4.13} and \eqref{4.16} we get:
		\begin{equation}\label{4.18}
			\left\{
			\begin{array}{ccc}
				\left\| u_n \right\|_p^p & = & \left\| v \right\|_p^p +\left\| w_n \right\|_p^p+\left\| u \right\|_p^p+o_n(1)\\
				E(u_n) & = & E(v)+E(w_n)+E(u)+o_n(1).
			\end{array}
			\right.
		\end{equation}
		Thus $0\leq \displaystyle \lim_{n\rightarrow \infty} \left\| w_n \right\|_p^p= \tau-b-d=l$.\\
		Case 1: $l>0$.
		Clearly, by \eqref{4.17} and \eqref{4.14}, one can see that $w_n \in \Lambda(\left\| w_n \right\|_p^p)$ for all $n\in \mathbb{N}$ and hence
		\begin{equation*}
			E(w_n) \geq m(\left\| w_n \right\|_p^p).
		\end{equation*}
		Thus by \eqref{4.18}, \autoref{Lemma 4.4} and \autoref{Lemma 4.5}
		\begin{eqnarray*}
			m(\tau) & = & E(u_n)+o_n(1)= E(v)+E(w_n)+E(u)+o_n(1)\\
			& \geq & E(v)+E(u)+m(\left\| w_n \right\|_p^p)+o_n(1) \geq m(b)+m(d) +m(\left\| w_n \right\|_p^p)+o_n(1)\\
			& = & m(b)+m(d)+m(l) \geq m(\tau),
		\end{eqnarray*}
		hence $m(b)$ and $m(d)$ are achieved by $v$ and $u$ respectively. Therefore, by \autoref{Lemma 4.4} we get:
		$$m(\tau) \geq m(b)+m(d)+m(l)> m(b+d)+m(l)\geq m(b+d+l)=m(\tau),$$
		now this is a contradiction.\\
		Case 2: $l=0$.\\
		Now, by \eqref{G_N} we have:
		$$|A(w_n)| \leq C_N \left\| \nabla w_n \right\|_p^{2q\gamma_q}\left\| w_n \right\|_p^{2q(1-\gamma_q)}\leq C' \left\| w_n \right\|_p^{2q(1-\gamma_q)}\rightarrow 0 \text{ as } n \rightarrow \infty,$$
		since $\{w_n\}$ is bounded in $W^{1,p}(\mathbb{R}^N)$ and $\displaystyle \lim_{n\rightarrow \infty}\left\| w_n \right\|_p^p=l=0.$ Thus by \eqref{S} we get
		\begin{eqnarray*}
			E(w_n) & = & \frac{T(w_n)}{p}-\frac{\left\| w_n \right\|_{p^*}^{p^*}}{p^*}-\frac{A(w_n)}{2q} \geq T(w_n)\left(\frac{1}{p}-\frac{T(u)^{\frac{p^*-p}{p}}}{p^*S^{\frac{p^*}{p}}}\right)+o_n(1)\\
			& > & T(w_n)\left(\frac{1}{p}-\frac{r_0^{\frac{p^*-p}{p}}}{p^*S^{\frac{p^*}{p}}}\right)+o_n(1)> T(w_n)h_{\tau_0}(r_0)+o_n(1) \text{ for every } n\in \mathbb{N},
		\end{eqnarray*}
		and hence
		\begin{equation}\label{4.19}
			\lim_{n\rightarrow \infty}E(w_n) \geq 0.
		\end{equation}
		Proceeding as in case 1, we have:
		\begin{equation}\label{4.20}
			m(\tau)  =  E(u_n)+o_n(1) \geq E(v)+E(u)+E(w_n)+o_n(1) \geq m(b)+m(d)\geq m(\tau),
		\end{equation}
		since $l=0$, \eqref{4.20} tells us that $m(b)$ and $m(d)$ are achieved and hence by \autoref{Lemma 4.4} we have:
		$$m(\tau)\geq m(b)+m(d)>m(b+d)=m(\tau).$$
		This leads us to a contradiction. Thus, we have
		\begin{equation*}
			\limsup_{n\rightarrow \infty} \sup_{z\in \mathbb{R}^N} \int_{B_1(z)}|v_n|^pdx = 0.
		\end{equation*}
		Therefore, by Lemma 1.1 of \cite{Lions1984concentration} we get $\{v_n\}\rightarrow 0$ in $L^r(\mathbb{R}^N)$ for all $r\in (p,p^*)$ and hence by \autoref{prop1.1} we have
		\begin{equation}\label{4.21}
			A(v_n) \leq \mu A_{\alpha}C_N \left(\int_{\mathbb{R}^N}|v_n|^{\frac{2Nq}{N+\alpha}}\right)^{\frac{N+\alpha}{N}}\rightarrow 0 \text{ as } n\rightarrow \infty, \text{ for all } q\in \left(\frac{p}{2}\left(\frac{N+\alpha}{N}\right), \bar{q_s}\right).
		\end{equation}
		Claim 3: $\displaystyle \lim_{n\rightarrow \infty} \left\| v_n \right\|_p^p=0$.\\
		Clearly, if $d=\tau$, then by \eqref{4.13}, $\left\| v_n\right\|_p^p=\left\| u_n \right\|_p^p-\left\| u \right\|_p^p+o_n(1)=\tau-d+o_n(1)=o_n(1)$ and we are done. Now, let  if possible, $d \neq \tau$, without loss of generality, we can assume $d<\tau$.  By, \eqref{4.14}, one can see that $v_n\in \Lambda(\left\| v_n \right\|_p^p)$ and hence $E(v_n)\geq m(\left\| v_n \right\|_p^p)$ for all $n\in \mathbb{N}$. Thus, by \eqref{4.16},  \autoref{Lemma 4.5} and \autoref{Lemma 4.4} we get:
		\begin{eqnarray*}
			m(\tau) & = & E(u_n)+o_n(1)= E(v_n)+E(u)+o_n(1) \geq m(\left\| v_n \right\|_p^p)+E(u) +o_n(1)\nonumber\\
			& = & m(\tau-d)+E(u) +o_n(1) \text{ for all } n\in \mathbb{N},
		\end{eqnarray*}
		and hence
		\begin{equation}\label{4.22}
			m(\tau) \geq m(\tau-d)+E(u).
		\end{equation}
		Since, $u\in \Lambda(d)$, we have $E(u)\geq m(d)$. Now, if $E(u)>m(d)$, then by \eqref{4.22} and \autoref{Lemma 4.4} we get
		$$m(\tau)>m(\tau-d)+m(d)\geq m(\tau).$$
		Therefore, by contradiction, we must have $E(u)=m(d)$. Now, since $m(d)$ is achieved, then again by \eqref{4.22} and \autoref{Lemma 4.4} we get :
		$$m(\tau)\geq m(\tau-d)+E(u)=m(\tau-d)+m(d)>m(\tau).$$
		Thus, our assumption that $d<\tau$ was wrong, similarly we can see that $d>\tau$ cannot hold and hence $\displaystyle \lim_{n\rightarrow \infty} \left\| v_n \right\|_p^p=0$.\\
		Claim 4 : $\displaystyle \lim_{n\rightarrow \infty} T(v_n)=0$.\\
		Clearly, if claim 4 holds, then by claim 3 and 4, we can see that $\{v_n\}\rightarrow 0$ in $W^{1,p}(\mathbb{R}^N)$ and hence $u_n(.-y_n)\rightarrow u$. Thus by continuity of $E$, we will get $E(u)=m(\tau)$. Therefore, proving claim 4 will suffice.\\
		Now, since $d=\tau$, we have $E(u)\geq m(\tau)$, also,
		$$m(\tau)=E(u_n)+o_n(1)=E(u)+E(v_n)+o_n(1),$$
		thus,
		$$E(v_n)=m(\tau)-E(u)+o_n(1)\leq o_n(1),$$
		hence \begin{equation}\label{4.23}
			\limsup_{n\rightarrow \infty}E(v_n)\leq 0.
		\end{equation}
		By \autoref{Lemma 4.1} we know that $h_{\tau_0}(r_0)=0$, also, $T(v_n)<r_0$, this implies
		$$\frac{T(v_n)}{p}-\frac{\left\| v_n \right\|_{p^*}^{p^*}}{p^*}\geq \frac{T(v_n)}{p}-\frac{\left\| \nabla v_n \right\|_p^p}{p^*S^{\frac{p^*}{p}}}\geq T(v_n)\left(\frac{1}{p}-\frac{T(v_n)^{\frac{p^*-p}{p}}}{p^*S^{\frac{p^*}{p}}}\right)>\beta_0T(v_n),$$
		where $$\beta_0=\left(\frac{1}{p}-\frac{r_0^{\frac{p^*-p}{p}}}{p^*S^{\frac{p^*}{p}}}\right)=h_{\tau_0(r_0)}+ \frac{C_N\tau_0^{\frac{2q(1-\gamma_q)}{p}}r_0^{\frac{2q\gamma_q-p}{p}}}{2q}>0.$$ Thus,
		$$\limsup_{n\rightarrow \infty} \beta_0T(v_n)\leq \limsup_{n\rightarrow \infty}\left(\frac{T(v_n)}{p}-\frac{\left\| v_n \right\|_{p^*}^{p^*}}{p^*}\right)=\limsup_{n\rightarrow \infty}\left(E(v_n)+\frac{A(v_n)}{2q}\right)\leq 0,$$
		by \eqref{4.21} and \eqref{4.23}. Therefore, $\displaystyle \lim_{n\rightarrow \infty}T(v_n)=0$ and we are done.
	\end{proof}
	\noindent Using the above proposition and properties of $m(\tau)$, let us prove the main Existence result for $\frac{p}{2}\left(\frac{N+\alpha}{N}\right)<q<\bar{q_s}$.
	\begin{myproof}{Theorem}{\ref{Theorem 1.2}}
		Let $\{u_n\}\subset \Lambda(\tau)$ be a minimizing sequence for $m(\tau)$. By \autoref{Lemma 4.2} we have $m(\tau)<0$, now if $\displaystyle \limsup_{n\rightarrow \infty} \left(\sup_{z\in \mathbb{R}^N}\int_{B_1(z)}|u_n|^2dx\right)=0$, then as done in \autoref{prop 4.1}, we will get $\{u_n\}\rightarrow 0$ in $L^r(\mathbb{R}^N)$ for all $r\in (p,p^*)$ and hence $\{A(u_n)\}\rightarrow 0$ as $n\rightarrow \infty$. Further, proceeding as in \eqref{4.19}, we will end up getting $m(\tau)\geq 0$, which is a contradiction. Thus
		$$\limsup_{n\rightarrow \infty} \left(\sup_{z\in \mathbb{R}^N}\int_{B_1(z)}|u_n|^2dx\right) \neq 0.$$
		Now, \autoref{prop 4.1} gives us the existence of a sequence $\{y_n\}$ in $\mathbb{R}^N$ and $u_{\tau}\in \Lambda(\tau)$ such that
		$$\{u_n(.-y_n)\}\rightarrow u_{\tau} \text{ in } W^{1,p}(\mathbb{R}^N) \text{ as } n\rightarrow \infty,$$
		and $E(u)=m(\tau)$. Now, let $u_{\tau}^*$ be the symmetric decreasing rearrangement of $u_{\tau}$. Using classical rearrangement inequalities \cite{Giacomoni2024Normalized}, we have :
		\begin{eqnarray*}
			[u^*_{\tau}]_{s,p}^p & = & \int_{\mathbb{R}^N} \int_{\mathbb{R}^N}\frac{|u_{\tau}^*(x)-u_{\tau}^*(y)|^{p}}{|x-y|^{N+sp}}dxdy = \int_{\mathbb{R}^N}\frac{1}{|z|^{N+sp}}\left(\int_{\mathbb{R}^N}|u_{\tau}^*(x+z)-u_{\tau}^*(x)|^pdx\right)dz\\
			& \leq & \int_{\mathbb{R}^N}\frac{1}{|z|^{N+sp}}\left(\int_{\mathbb{R}^N}\frac{|u_{\tau}(x+z)-u_{\tau}(x)|^p}{|x-y|^{N+sp}}dx\right)dz=[u_{\tau}]_{s,p}^p,
		\end{eqnarray*}
		$\left\| u_{\tau}^*\right\|_{p^*}^{p^*}=\left\| u_{\tau}\right\|_{p^*}^{p^*}\;;\; \left\| \nabla u_{\tau}^*\right\|_p^p\leq \left\| \nabla u_{\tau} \right\|_p^p \text{ and } A(u_{\tau}^*)\geq A(u_{\tau}).$ Therefore, $T(u_{\tau}^*)\leq T(u_{\tau})<r_0$, hence $E(u_{\tau}^*)\geq m(\tau)$, since $u_{\tau}^*\in \Lambda(\tau)$. Thus,
		$$m(\tau)\leq E(u_{\tau}^*)=\frac{T(u_{\tau}^*)}{p}-\frac{\left\| u_{\tau}^*\right\|_{p^*}^{p^*}}{p^*}-\frac{A(u_{\tau}^*)}{2q}\leq \frac{T(u_{\tau})}{p}-\frac{\left\| u_{\tau}\right\|_{p^*}^{p^*}}{p^*}-\frac{A(u_{\tau})}{2q}=m(\tau).$$
		Above calculations tell us that $E(u_{\tau}^*)=m(\tau)$, and hence $m(\tau)$ is achieved by a radially symmetric function in $W^{1,p}(\mathbb{R}^N)$. Moreover, since $E(u_{\tau}^*)=m(\tau)<0$ and $u_{\tau}^*$ solves \eqref{1.1} for  $\lambda=\lambda_\tau$, we get
		\begin{eqnarray*}
			\lambda_\tau \left\| u_{\tau}^*\right\|_p^p & = &T(u_{\tau}^*)-A(u_{\tau}^*)-\left\| u_{\tau}^*\right\|_{p^*}^{p^*} = pm(\tau)+\frac{pA(u_{\tau}^*)}{2q}+\frac{p\left\| u_{\tau}^*\right\|_{p^*}^{p^*}}{p^*}-A(u_{\tau}^*)-\left\| u_{\tau}^*\right\|_{p^*}^{p^*}\\
			& = & pm(\tau) +\left(\frac{p}{2q}-1\right)A(u_{\tau}^*)+\left(\frac{p}{p^*}-1\right)\left\| u_{\tau}^* \right\|_{p^*}^{p^*}<0,
		\end{eqnarray*}
		therefore, $\lambda_\tau<0$.
	\end{myproof}
	
	\section{Case  \texorpdfstring{$2$}{TEXT} 
		: \texorpdfstring{$\bar{q}_s \leq q < \frac{p}{2}\left(\frac{N+\alpha}{N-p}\right)$}
		{TEXT} }
	Let $S_r(\tau)=S(\tau)\cap W_r^{1,p}(\mathbb{R}^N)$, $P_r(\tau)=P(\tau)\cap W_r^{1,p}(\mathbb{R}^N)$ and $m_r(\tau)=\displaystyle \inf_{u\in P_r(\tau)}E(u)$. Define $\Phi:W^{1,p}(\mathbb{R}^N)\times \mathbb{R}\rightarrow \mathbb{R}$ as
	$$\Phi (u,t):= E(t\star u )=\frac{e^{pt}}{p}\left\| \nabla u \right\|_p^p+\frac{e^{spt}}{p}[u]_{s,p}^p-\frac{e^{p^*t}}{p^*}\left\| u \right\|_{p^*}^{p^*}-\frac{e^{2q\gamma_qt}}{2q}A(u),$$
	clearly, $\Phi \in C^1$.
	\begin{lemma}\label{Lemma 5.1}
		Let $\tau>0$, then for any fixed $u\in S_r(\tau)$, we have:
		\begin{enumerate}
			\item $T(t\star u )\rightarrow 0$ and $\Phi (u,t)\rightarrow 0^+$ as $t\rightarrow -\infty$,
			\item $T(t\star u)\rightarrow +\infty$ and $\Phi(u,t)\rightarrow -\infty$ as $t\rightarrow +\infty$.
		\end{enumerate}
		In particular, for $q=\bar{q}_s$, the above result holds for all $\tau\in (0,\bar{\tau_s})$ where $\bar{\tau_s}=(\frac{2\bar{q}_s}{pC_N})^{\frac{N}{\alpha+sp}}$.
	\end{lemma}
	\begin{proof}
		Since, $$T(t\star u )=e^{pt}\left\| \nabla u \right\|_p^p+e^{spt}[u]_{s,p}^p,$$ clearly, $T(t\star u )\rightarrow 0$ as $t\rightarrow -\infty$ and $T(t\star u)\rightarrow +\infty$ as $t\rightarrow +\infty$. For $q=\bar{q}_s$, by \eqref{fractionalGN} we have
		\begin{eqnarray*}
			\Phi (u,t) & = & \frac{e^{pt}}{p}\left\| \nabla u \right\|_p^p+\frac{e^{spt}}{p}[u]_{s,p}^p-\frac{e^{p^*t}}{p^*}\left\| u \right\|_{p^*}^{p^*}-\frac{e^{spt}}{2\bar{q}}A(u)\\
			& \geq & \frac{e^{spt}}{p}[u]_{s,p}^p-\frac{e^{spt}C_N}{2\bar{q}}[u]_{s,p}^p\tau^{\frac{\alpha+sp}{N}}-\frac{e^{p^*t}}{p^*}\left\| u\right\|_{p^*}^{p^*}\\
			& = & e^{spt}[u]_{s,p}\left(\frac{1}{p}-\frac{C_N \tau^{\frac{\alpha+sp}{N}}}{2\bar{q}}\right)-\frac{e^{p^*t}}{p^*}\left\| u\right\|_{p^*}^{p^*}.
		\end{eqnarray*}
		Therefore, 
		$$\Phi(u,t) \rightarrow \left\{ \begin{array}{cc}
			0^+ & \text{ as } t\rightarrow -\infty,\\
			-\infty & \text{ as } t\rightarrow +\infty,
		\end{array}
		\right.
		$$
		for all $\tau< \bar{\tau_s}$. Also, since for  $\bar{q}_s<q<\frac{p}{2}(\frac{N+\alpha}{N-p})$, we have $sp<2q\gamma_q $ and $p<p^*$ we get the required result.
	\end{proof}
	\noindent Define $$A_k:=\{ u\in S_r(\tau): T(u)<k\}$$
	and 
	$$\partial A_k:=\{ u\in S_r(\tau): T(u)=k\}.$$
	\begin{lemma}\label{Lemma 5.2}
		There exists $k_2>k_1>0$ such that 
		\begin{equation}\label{5.1}
			0< \sup_{u\in A_{k_1}}E(u)<\inf_{u\in \partial A_{k_2}}E(u),
		\end{equation}
		and $E(u)$, $M(u)>0$ for all $u\in A_{k_2}$, whenever $\bar{q}_s<q<\frac{p}{2}(\frac{N+\alpha}{N-p})$ with $\tau>0$ and $q=\bar{q}_s$ with $\tau\in (0,\bar{\tau_s})$.
	\end{lemma}
	\begin{proof}
		For $q=\bar{q}_s$, by \eqref{fractionalGN} and \eqref{S} we have:
		\begin{equation*}
			E(u)  \geq  \frac{T(u)}{p}-\frac{T(u)^{\frac{p^*}{p}}}{p^*S^{\frac{p^*}{p}}}-\frac{C_NT(u)\tau^{\frac{\alpha+sp}{N}}}{2\bar{q}}=\frac{T(u)}{p}\left(1-\frac{pC_N\tau^{\frac{\alpha+sp}{N}}}{2\bar{q}}\right)-\frac{T(u)^{\frac{p^*}{p}}}{p^*S^{\frac{p^*}{p}}},
		\end{equation*}
		and 
		\begin{equation*}
			M(u) \geq sT(u)-\frac{T(u)^{\frac{p^*}{p}}}{S^{\frac{p^*}{p}}}-\gamma_{\bar{q}}C_NT(u)\tau^{\frac{\alpha+sp}{N}}= s\left(1-\frac{pC_N\tau^{\frac{\alpha+sp}{N}}}{2\bar{q}}\right)T(u)-\frac{T(u)^{\frac{p^*}{p}}}{S^{\frac{p^*}{p}}}.
		\end{equation*}
		Now, since $p^*>p$, then for $\tau<\bar{\tau_s}$, we can find $k_2>0$ small enough so that $E(u)\geq \rho>0$ for all $u\in \partial A_{k_2}$, $M(u)$, $E(u)>0$ for all $u\in A_{k_2}$ and 
		\begin{equation*}
			\inf_{u\in \partial A_{k_2}}E(u)\geq\rho>0.
		\end{equation*}
		Also, since $$E(u)\leq \frac{T(u)}{p},$$
		we can find $0<k_1<k_2$ such that 
		$$0<\sup_{u\in A_{k_1}} E(u) <\rho \leq \inf_{u\in \partial A_{k_2}}E(u).$$
		Now, for $\bar{q}_s<q<\frac{p}{2}\left(\frac{N+\alpha}{N-p}\right)$, we have:
		\begin{equation*}
			E(u)  \geq  \frac{T(u)}{p}-\frac{T(u)^{\frac{p^*}{p}}}{p^*S^{\frac{p^*}{p}}}-\frac{C_N}{2q}T(u)^{\frac{2q\gamma_q}{ps}}\tau^{\frac{2q(1-\gamma_q)}{sp}},
		\end{equation*}
		$$M(u)\geq s T(u)-\frac{T(u)^{\frac{p^*}{p}}}{S^{\frac{p^*}{p}}}-C_N\gamma_qT(u)^{\frac{2q\gamma_q}{sp}}\tau^{\frac{2q(1-\gamma_q)}{sp}},$$
		and 
		$$E(u)\leq \frac{T(u)}{p}.$$
		Since $p^*>p$ and $2q\gamma_q>sp$, following the same procedure as done for the case $q=\bar{q}_s$, we can find $0<k_1<k_2$ satisfying \eqref{5.1} and $E(u)$, $M(u)>0$ for all $u\in A_{k_2}$.
	\end{proof}
	\noindent By \autoref{Lemma 5.1} and \autoref{Lemma 5.2}, we can find $u_1$, $u_2\in S_r(\tau)$ such that 
	\begin{equation}\label{u_1,u_2}
		T(u_1)\leq k_1< k_2<T(u_2) ;\;\; E(u_1)>0>E(u_2) \text{ and } M(u_2)<0.
	\end{equation}
	Define 
	$$r(\tau):= \inf_{\eta \in \Gamma(\tau)}\max_{z\in [0,1]}E(\eta(z)),$$
	where 
	$$\Gamma(\tau)=\{\eta\in C([0,1],S_r(\tau)): \eta(0)=u_1 \text{ and } \eta(1)=u_2\},$$
	and $$\tilde{r}(\tau):=\inf_{\tilde{\eta}\in \tilde{\Gamma}(\tau)}\max_{z\in [0,1]} \Phi(\tilde{\eta}(z)),$$
	where 
	$$\tilde{\Gamma}(\tau)=\{\tilde{\eta}\in C([0,1],S_r(\tau)\times\mathbb{R}): \tilde{\eta}(0)=(u_1,0) \text{ and } \tilde{\eta}(1)=(u_2,0)\}.$$
	Clearly, $r(\tau)\geq \max\{E(u_1),E(u_2)\}:=\sigma_{\tau}$.
	\begin{lemma}\label{Lemma 5.3}
		For any fixed $u\in S_r(\tau)$, the function $I_u(t)$ has a unique critical point $t_u\in \mathbb{R}$ such that $t_u\star u \in P(\tau)$ corresponding to the maxima of $I_u$. In particular, for $q=\bar{q}_s$, the above statement is true for all $\tau \in (0,\bar{\tau_s})$.
	\end{lemma}
	\begin{proof}
		Let $u\in S_r(\tau)$ and $q=\bar{q}_s$. By \autoref{Lemma 3.2}, we know that $t_u\in \mathbb{R}$ is a critical point of $I_u$ if and only if $t_u\star u \in P_r(\tau)$. Now, since
		$$I_u'(t)= e^{pt}\left\| \nabla u \right\|_p^p+se^{spt}[u]_{s,p}^p-e^{p^*t}\left\| u \right\|_{p^*}^{p^*}-\gamma_{\bar{q}_s}e^{spt}A(u),$$
		if $t_u$ is a critical point of $I_u$, then $g(t_u)=\gamma_{\bar{q}_s}A(u)=C_u$ where
		$$g(t)= e^{(p-sp)t_u}\left\| \nabla u \right\|_p^p+s[u]_{s,p}^p-e^{(p^*-sp)t_u}\left\| u \right\|_{p^*}^{p^*}.$$
		Claim : $I_u$ has unique critical point.\\
		Let if possible, $I_u$ has more than one critical points. Now, by \autoref{Lemma 5.1}, we know that $I_u(t)\rightarrow 0^+$ as $t\rightarrow -\infty$ and $I_u(t)\rightarrow -\infty$ as $t\rightarrow +\infty$, one can easily see that in this case, $I_u$ must have atleast three critical. 
		Then $g$ must attain $C_u$ at more than two points, hence $g$ must have atleast two critcal points. But
		$$t_0=\frac{1}{p^*-p}\ln\left(\frac{p(1-s)\left\| \nabla u \right\|_p^p}{(p^*-sp)\left\| u \right\|_{p^*}^{p^*}}\right),$$
		is the only critical point of $g$. Therefore, by contradiction, $I_u$ has exactly one critical point $t_u$ corresponding to its global maxima.\\
		Similarly, for $\bar{q}_s<q<\frac{p}{2}(\frac{N+\alpha}{N-p})$, if $I_u$ has more than one critical points, then $\tilde{g}$ defined as:
		$$\tilde{g}(t)= e^{(p-2q\gamma_q)}\left\| \nabla u \right\|_p^p+se^{(sp-2q\gamma_q)t}[u]_{s,p}^p-e^{(p^*-2q\gamma_q)t}\left\| u \right\|_{p^*}^{p^*}$$
		must attain $\gamma_qA(u)=\tilde{C}_u$ at more than two points and hence have atleast two critical points. That is, 
		$$\tilde{g}'(t)= e^{(p-2q\gamma_q)t}\left((p-2q\gamma_q)\left\| \nabla u \right\| _p^p+s(sp-2q\gamma_q)e^{(sp-p)t}[u]_{s,p}^p-(p^*-2q\gamma_q)e^{(p^*-p)t}\left\| u \right\|_{p^*}^{p^*}\right),$$
		has atleast two roots,  hence 
		$$f(t)=(p-2q\gamma_q)\left\| \nabla u \right\| _p^p+s(sp-2q\gamma_q)e^{(sp-p)t}[u]_{s,p}^p-(p^*-2q\gamma_q)e^{(p^*-p)t}\left\| u \right\|_{p^*}^{p^*}$$
		has atleast one crtical point. Now, if $t_0$ is a critical point of $f$, then we must have
		$$(p^*-2q\gamma_q)=\frac{s(sp-2q\gamma_q)(sp-p)e^{(sp-p)t}[u]_{s,p}^p}{(p^*-p)e^{(p^*-p)t}\left\| u \right\|_{p^*}^{p^*}}>0 \text{ for all } \bar{q}<q<\frac{p}{2}\left(\frac{N+\alpha}{N-p}\right),$$
		that is, 
		$$2q\gamma_q<p^* \text{ for all } \bar{q}_s<q<\frac{p}{2}\left(\frac{N+\alpha}{N-p}\right),$$
		hence we must have $\frac{N+\alpha}{N-p}<\frac{N+\alpha+p^*}{N}$, but since  $p^*=\frac{N}{p}(p^*-p)$, we have:
		$$(N-p)(p^*+N+\alpha)-N(N+\alpha)=N(p^*-p)-\alpha p-p^*p=-\alpha p<0.$$
		Therefore, $f$ does not have any critical point and hence $I_u$ has unique critical point, which corresponds to its maxima.
	\end{proof}
	\begin{lemma}\label{Lemma 5.4}
		$r(\tau)=\tilde{r}(\tau)=m_r(\tau)$.
	\end{lemma}
	\begin{proof}
		Claim 1 : $r(\tau)=\tilde{r}(\tau)$.\\
		For any $(\tilde{\eta}_1,\tilde{\eta}_2)=\tilde{\eta}\in \tilde{\Gamma}(\tau)$, define $g(z):=\tilde{\eta}_2(z)\star \tilde{\eta}_1(z)$ for every $z\in [0,1]$. Clearly, $g\in \Gamma(\tau)$ and hence
		$$r(\tau)\leq \max_{z\in [0,1]}E(g(z))=\max_{z\in [0,1]}\Phi (\tilde{\eta}(z)),$$
		since, $\tilde{\eta}\in \Gamma(\tau)$ is arbitrary, we get $r(\tau)\leq \tilde{r}(\tau)$. On the other hand, for any $\eta\in \Gamma(\tau)$, we can take $\tilde{\eta}=(\tilde{\eta}_1,0)\in \tilde{\Gamma}(\tau)$ and deduce that:
		$$\tilde{r}(\tau)\leq \max_{z\in [0,1]}\Phi(\tilde{\eta}(z))=\max_{z\in [0,1]}E(\eta(z)),$$
		hence $\tilde{r}(\tau)\leq r(\tau)$. Thus we are done with claim 1.\\
		Claim 2 : $r(\tau)=m_r(\tau)$.\\
		For any $u\in P_r(\tau)$, by \autoref{Lemma 5.1} we can find $t_1<0<t_2$ such that $T(t_1\star u) <k_1<k_2<T(t_2\star u)$, $\Phi(u,t_2)<0<\Phi(u,t_1)$ and $M(t_2\star u)=I_u'(t_2)<0$. By the definition of $u_1$ and $u_2$ given in \eqref{u_1,u_2}, we can take $u_1=t_1 \star u$ and $u_2=t_2\star u$. Setting $\eta(z):=((1-z)t_1+zt_2)\star u$ for all $z\in [0,1]$, clearly $\eta\in \Gamma(\tau)$ and since $I_u'(0)=0$ by \autoref{Lemma 5.3} we get
		\begin{equation}\label{5.3}
			r(\tau)\leq \max_{z\in [0,1]}E(\eta(z))=\max_{t\in [t_1,t_2]}E(t\star u)=\max_{t\in [t_1,t_2]}I_u(t)=I_u(0)=E(u).
		\end{equation}
		as \eqref{5.3} is true for every $u\in P_r(\tau)$, we can deduce that $r(\tau)\leq m_r(\tau)$.\\
		Now, if we define $\tilde{h}(z):= M(\tilde{\eta}_2(z)\star \tilde{\eta}_1(z))$ on $[0,1]$, for some $(\tilde{\eta}_1,\tilde{\eta}_2)=\tilde{\eta}\in \tilde{\Gamma}(\tau)$, we get
		$$\tilde{h}(0)= M(0\star u_1)=M(u_1)>0,$$
		by \autoref{Lemma 5.2}, and 
		$$\tilde{h}(1)=M(0\star u_2)=M(u_2)<0.$$
		Thus, there exists $z_0\in (0,1)$ such that $\tilde{h}(z_0)=0$ and hence $\tilde{\eta}_2(z_0)\star \tilde{\eta}_1(z_0)\in P_r(\tau)$. Therefore,
		$$m_r(\tau)\leq E(\tilde{\eta}_2(z_0)\star \tilde{\eta}_1(z_0))=\Phi(\tilde{\eta}(z_0))\leq \max_{z\in [0,1]}\Phi (\tilde{\eta}(z_0)).$$
		Since $\tilde{\eta}\in \tilde{\Gamma}(\tau)$ is arbitrary, we are done. 
	\end{proof}
	\begin{lemma}\label{Lemma 5.5}
		For $u\in S_r(\tau)$, defining
		$$T_{u}=\{w\in W^{1,p}(\mathbb{R}^N): \int_{\mathbb{R}^N}|u|^{p-2}uw=0\},$$
		then we can find a sequence $\{u_n\}\subset S_r(\tau)$ such that
		\begin{enumerate}
			\myitem{1}\label{P1} $\displaystyle \lim_{n\rightarrow \infty}E(u_n) = m_r(\tau)$,
			\myitem{2}\label{P2} $\displaystyle \lim_{n\rightarrow \infty} M(u_n)=0$,
			\myitem{3} \label{P3} $\displaystyle \lim_{n\rightarrow \infty} E'_{S_r(\tau)}=0$, that is $E'(u_n)(w)\rightarrow 0$ uniformly for all $w\in T_{u_n}$ with $\left\| w \right\|\leq 1$ .
		\end{enumerate}
	\end{lemma}
	\begin{proof}
		Following the proof of Lemma 2.3 of \cite{Jeanjean1997Existence}, we can find the sequence $\{(v_n,t_n)\}\subset S_r(\tau)\times \mathbb{R}$ such that
		\begin{equation}\label{5.4}
			\{\Phi (v_n,t_n)\}\rightarrow \tilde{r}(\tau),
		\end{equation}
		and 
		\begin{equation*}
			\left\| \Phi'_{S_r(\tau)\times \mathbb{R}} (v_n,t_n)\right\| \leq \frac{2}{\sqrt{n}}\rightarrow 0 \text{ as } n\rightarrow \infty,
		\end{equation*}
		that is,
		\begin{equation}\label{5.5}
			|\Phi'(v_n,t_n)(w_n)|\leq \frac{2}{\sqrt{n}}\left\| w \right\|_{\mathbb{W}} \text{ for all } w_n\in \tilde{T}_{(v_n,t_n)}
		\end{equation}
		where $\tilde{T}_{(v_n,t_n)}=\{(w_1,w_2)\in S_r(\tau)\times\mathbb{R}:\int_{\mathbb{R}^N}|v_n|^{p-2}v_nw_1=0\}$.
		Set $u_n=t_n \star v_n \in S_r(\tau)$. Thus we get \ref{P1} by \eqref{5.4} and \autoref{Lemma 5.4}. Now, for any $w=(w_1,w_2)\in \mathbb{W}$, we have:
		\begin{eqnarray*}
			\Phi'(v_n,t_n)(w_1,w_2) & = & e^{pt_n}\int_{\mathbb{R}^N}|\nabla v_n|^{p-2}\nabla v \nabla w_1+e^{pt_n}w_2\left\| \nabla v_n \right\|_p^p +e^{spt_n}\ll v_n,w_1\gg\\
			&& +sw_2e^{spt_n}[v_n]_{s,p}^p-e^{p^*t_n}\int_{\mathbb{R}^N}|v_n|^{p^*-2}v_nw_1-w_2e^{p^*t_n}\left\| v_n \right\|_{p^*}^{p^*}\\
			&& -e^{2q\gamma_qt_n}\mu\int_{\mathbb{R}^N}(I_{\alpha}*|v_n|^{q})|v_n|^{q-2}v_nw_1-\gamma_qw_2e^{2q\gamma_qt_n}A(v_n),
		\end{eqnarray*}
		taking $w=(0,1)$ we get $w\in \tilde{T}_{(v_n,t_n)}$ and hence by \eqref{5.5}
		\begin{eqnarray*}
			|M(u_n)| & = & |M(t_n\star v_n)| =|e^{pt_n}\left\| \nabla v_n\right\|_p^p+se^{spt_n}[v_n]_{s,p}^p-e^{p^*t_n}\left\| v_n\right\|_{p^*}^{p^*}-\gamma_qe^{2q\gamma_qt_n}A(v_n)|\\
			& = & |\Phi'(v_n,t_n)(w)|\leq \frac{2}{\sqrt{n}}\rightarrow 0 \text{ as } n\rightarrow \infty.
		\end{eqnarray*}
		Thus we are done with \ref{P2}. Further, taking $w\in W^{1,p}(\mathbb{R}^N)$ such that $\int_{\mathbb{R}^N}|u_n|^{p-2}u_nw=0$, we get:
		\begin{eqnarray*}
			E'(u_n)(w) & = & \int_{\mathbb{R}^N}|\nabla u_n|^{p-2}\nabla u_n\nabla w+\ll u_n,w\gg-\int_{\mathbb{R}^N}|u_n|^{p^*-2}u_nw\\
			& & -\mu\int_{\mathbb{R}^N}(I_{\alpha}*|u_n|^q)|u_n|^{q-2}u_nw\\
			& = & e^{pt_n}\int_{\mathbb{R}^N}|\nabla v_n|^{p-2}\nabla v_n \nabla\tilde{w}+e^{spt_n}\ll v_n, \tilde{w}\gg -e^{p^*t_n}\int_{\mathbb{R}^N}|v_n|^{p^*-2}v_n\tilde{w}\\
			&& -e^{2q\gamma_qt_n}\mu\int_{\mathbb{R}^N}(I_{\alpha}*|v_n|^q)|v_n|^{q-2}v_n\tilde{w}=\Phi'(v_n,t_n)(\tilde{w},0),
		\end{eqnarray*}
		where $\tilde{w}(x)=e^{-\frac{Nt_n}{p}}w(e^{-t_n}x)$. Clearly, $(\tilde{w},0)\in \tilde{T}_{(v_n,t_n)}$, therefore, by \eqref{5.5} we get:
		$$|E'(u_n)(w)|=|\Phi'(v_n,t_n)(\tilde{w},0)|
		\leq \frac{2}{\sqrt{n}}\left\|(\tilde{w},0) \right\|_{\mathbb{W}}\leq\frac{C}{\sqrt{n}}\left\| w\right\| \text{ for large } n.$$
		Since, $w\in T_{u_n}$ is arbitrary, 
		$$\sup\{E'(u_n)(w): w\in T_{u_n} \text{ with } \left\| w \right\|\leq 1\}\leq \frac{C}{\sqrt{n}}\rightarrow 0 \text{ as } n\rightarrow \infty.$$
	\end{proof}
	\begin{lemma}\label{Lemma 5.6}
		If $0 \neq m_r(\tau)<\frac{S^{\frac{N}{p}}}{N}$ and $\{u_n\}$ is the sequence deduced in \autoref{Lemma 5.5}, then for large enough $\mu>0$, either of the following holds true:
		\begin{enumerate}
			\myitem{a)}\label{Lemma 2.2.6a} Upto a subsequence, $\{u_n\}\rightharpoonup u$ weakly in $W^{1,p}(\mathbb{R}^N)$ but not strongly and $u$ is a solution of \eqref{1.1} for some parameter $\lambda \in \mathbb{R}$ with
			$$E(u)\leq m_r(\tau)-\frac{S^{\frac{N}{p}}}{N}.$$
			\myitem{b)}\label{Lemma 2.2.6b} Upto a subsequence, $\{u_n\}\rightarrow u$ strongly in $W^{1,p}(\mathbb{R}^N)$ and $u$ solves \eqref{1.1}-\eqref{1.2} for some $\lambda<0$, with $E(u)=m_r(\tau)$.
		\end{enumerate}
		In particular, for $q=\bar{q}_s$, we will take $\tau\in (0,\bar{\tau_s})$ and for $\bar{q_s}<q\leq \bar{q}:=\frac{p}{2}(\frac{N+\alpha+p}{N})$, we take $\tau< \bar{\tau_q}=(\frac{2qp^*}{NC_N(p^*-2q\gamma_q)})^{\frac{p}{2q(1-\gamma_q)}}$.
	\end{lemma}
	\begin{proof}
		Let $\{u_n\}$ be the sequence deduced in \autoref{Lemma 5.5}.\\
		Claim 1: $\{u_n\}$ is bounded.\\
		For $q=\bar{q_s}$, since $E(u_n)\rightarrow m_r(\tau)$ and $M(u_n)\rightarrow 0$ then for large $n$, using \eqref{fractionalGN} we have
		\begin{eqnarray*}
			m_r(\tau)+1 & \geq & E(u_n) = E(u_n)-\frac{1}{p^*}M(u_n)+o_n(1)\\
			& = & \left(\frac{1}{p}-\frac{1}{p^*}\right)\left\| \nabla u_n\right\|_p^p+\left(\frac{1}{p}-\frac{s}{p^*}\right)[u_n]_{s,p}^p-\left(\frac{1}{2\bar{q}_s}-\frac{\gamma_{\bar{q}_s}}{p^*}\right)A(u_n)+o_n(1)\\
			& = & \frac{\left\| \nabla u_n \right\|_p^p}{N}+\left(\frac{1}{p}-\frac{s}{p^*}\right)[u_n]_{s,p}^p-\frac{(p^*-sp)}{2p^*\bar{q}_s}A(u_n)+o_n(1)\\
			& \geq & \frac{\left\| \nabla u_n \right\|_p^p}{N}+\left(\frac{1}{p}-\frac{s}{p^*}-\left(\frac{(p^*-sp)}{2p^*\bar{q}_s}\right)C_N\tau^{\frac{\alpha+sp}{N}}\right)[u_n]_{s,p}^p\geq KT(u_n),
		\end{eqnarray*}
		where $K=\min\{\frac{1}{N},\frac{1}{p}-\frac{s}{p^*}-\left(\frac{(p^*-sp)}{2p^*\bar{q}_s}\right)C_N\tau^{\frac{\alpha+sp}{N}}\}>0$ for all $\tau < \bar{\tau_s}$. \\
			Similarly, for $\bar{q_s}<q \leq \bar{q}$, using \eqref{G_N}, for large $n$ we get:
			\begin{eqnarray*}
				m_r(\tau)+1 & \ge & E(u_n)=E(u_n)-\frac{1}{p^*}M(u_n)+o_n(1)\\
				& \geq & \frac{T(u_n)}{N}+\frac{(2q\gamma_q-p^*)}{2qp^*}C_N\tau^{\frac{2q(1-\gamma_q)}{p}}T(u_n),
			\end{eqnarray*}
			since, $2q\gamma_q<p<p^*$ and $T(u_n)\geq 1$ for large $n$ (if not, then directly $\{u_n\}$ becomes bounded). Hence $\{u_n\}$ is bounded, for $\tau<\bar{\tau_q}$, and for $\bar{q}<q<\frac{p}{2}(\frac{N+\alpha}{N-p})$, we have:
			\begin{eqnarray*}
				m_r(\tau)+1 & \geq & E(u_n)=E(u_n)+\frac{M(u_n)}{p}+o_n(1)\\
				& \geq & \left(\frac{1}{p}-\frac{1}{p^*}\right)\left\| u\right\|_{p^*}^{p^*}+\left(\frac{\gamma_q}{p}-\frac{1}{2q}\right)A(u_n)+o_n(1).
			\end{eqnarray*} 
			This tells us that $\{\left\| u_n\right\|_{p^*}^{p^*}\}$ and $\{A(u_n)\}$ are bounded, thus
			$$sT(u_n)\leq \left\| \nabla u_n \right\|_p^p+s[u_n]_{s,p}^p=\left\| u_n\right\|_{p^*}^{p^*}+\gamma_q A(u_n)+o_n(1)<M,$$
			for some $M>0$. Therefore, $\{u_n\}$ is bounded in $S_r(\tau)\subset W_r^{1,p}(\mathbb{R}^N)$. Since $W_r^{1,p}(\mathbb{R}^N)$ is compactly imbedded in $L^t(\mathbb{R}^N)$ for all $t\in (p,p^*)$, we can find $u\in W_r^{1,p}(\mathbb{R}^N)$ such that upto a subsequence
			\begin{eqnarray*}
				\{u_n\} & \rightharpoonup & u, \text{ weakly in } W_r^{1,p}(\mathbb{R}^N),\\
				\{u_n\} & \rightarrow & u, \text{ strongly in } L^t(\mathbb{R}^N) \text{ for all } t\in (p,p^*),\\
				\{u_n\} & \rightarrow & u \text{ almost everywhere in } \mathbb{R}^N.
			\end{eqnarray*}
			Following the proof of Lemma 2.5 of \cite{Jeanjean1997Existence}, one can easily see that, for every $w\in W^{1,p}(\mathbb{R}^N)$
			\begin{equation}\label{E(u_n)}
				E'(u_n)(w)-\lambda_n\int_{\mathbb{R}^N}|u_n|^{p-2}u_nw= o_n(1)					
			\end{equation}
			where 
			$$\lambda_n=\frac{E'(u_n)(u_n)}{\left\| u_n\right\|_p^p}=\frac{1}{\tau}\left(\left\| \nabla u_n \right\|_p^p+[u_n]_{s,p}^p-\left\| u_n \right\|_{p^*}^{p^*}-A(u_n)\right).$$
			Since $\{u_n\}$ is bounded in $W^{1,p}(\mathbb{R}^N)$, clearly, $\{\lambda_n\}$ must be bounded and convergent upto a subsequence to some $\lambda_\tau\in \mathbb{R}$.  \\
			Claim 2: $u\neq 0$.\\
			Let if possible, $u(x)=0$ for all $x\in \mathbb{R}^N$. Then by \eqref{HLS}, we have
			\begin{equation}\label{5.7}
				A(u_n)\leq \mu C_N\left(\int_{\mathbb{R}^N}|u_n|^{\frac{2Nq}{N+\alpha}}\right)^{\frac{N+\alpha}{N}}\rightarrow \mu C_N\left(\int_{\mathbb{R}^N}|u|^{\frac{2Nq}{N+\alpha}}\right)^{\frac{N+\alpha}{N}}=0,
			\end{equation}
			and since, $M(u_n)=o_n(1)$, we get $\displaystyle \lim_{n\rightarrow \infty}(\left\| \nabla u_n \right\|_p^p+s[u_n]_{s,p}^p)=\lim_{n\rightarrow \infty} \left\| u_n \right\|_{p^*}^{p^*}$. Let $l\geq 0$ be such that $$l=\displaystyle \lim_{n\rightarrow \infty}(\left\| \nabla u_n \right\|_p^p+s[u_n]_{s,p}^p)=\lim_{n\rightarrow \infty} \left\| u_n \right\|_{p^*}^{p^*}.$$
			Then by \eqref{S}, we get: $$S\leq \frac{l}{l^{\frac{p}{p^*}}}\Rightarrow l(Sl^{\frac{p}{p^*}-1}-1)\leq 0,$$
			Therefore, either $l=0$ or $l\geq S^{\frac{N}{p}}$. Now, if $l\geq S^{\frac{N}{p}}$, then by \eqref{5.7}
			\begin{eqnarray*}
				m_r(\tau) & = & E(u_n) +o_n(1)=E(u_n)-\frac{1}{p^*}M(u_n)+o_n(1)\\
				&\geq & \frac{1}{N}\left\| \nabla u_n \right\|_p^p +\frac{s}{N}[u_n]_{s,p}^p+o_n(1)=\frac{l}{N}+o_n(1)\geq \frac{S^{\frac{N}{p}}}{N},
			\end{eqnarray*}
			but we are given that $m_r(\tau)<\frac{S^{\frac{N}{p}}}{N}$, and if $l=0$, then we will get $m_r(\tau)=\displaystyle \lim_{n\rightarrow \infty}E(u_n)=0$. Thus, by contradiction, we must have $u\neq 0$. Also, since $\{u_n\}\rightharpoonup u$, we can see that, 
			$$\int_{\mathbb{R}^N}|\nabla u|^{p-2}\nabla u\nabla \zeta+\ll u,\zeta \gg=\lambda \int_{\mathbb{R}^N}|u|^{p-2}u\zeta+\int_{\mathbb{R}^N}|u|^{p^*-2}u\zeta +\mu\int_{\mathbb{R}^N}(I_{\alpha}*|u|^{q})|u|^{q-2}u\zeta,$$
			for all $\zeta\in C_{c}^{\infty}(\mathbb{R}^N)$. Thus, $u$ is a weak solution of 
			\begin{equation}\label{u}
				-\Delta_pu +(-\Delta_p)^su= \lambda |u|^{p-2}u +\mu(I_{\alpha}*|u|^q)|u|^{q-2}u +|u|^{p^*-2}u \text{ in } \mathbb{R}^N,
			\end{equation}
			and hence by \autoref{Lemma 3.1} $M(u)=0$. Now, let $v_n=(u_n-u) \rightharpoonup 0$, weakly in $W^{1,p}(\mathbb{R}^N)$, then as done in \autoref{prop 4.1}, we have:
			\begin{equation}\label{5.9}
				\left\{ \begin{array}{cc}
					\left\| u_n \right\|_p^p & =  \left\| v_n \right\|_p^p+\left\| u \right\|_p^p+o_n(1)\\
					\left\| u_n \right\|_{p^*}^{p^*} & =  \left\| v_n \right\|_{p^*}^{p^*}+ \left\| u \right\|_{p^*}^{p^*}+o_n(1)\\
					T_s(u_n) & =  T_s(v_n) +T_s(u) +o_n(1) \\
					A(u_n) & = A(v_n)+A(u)+o_n(1)
				\end{array}
				\right.
			\end{equation}
			where  $T_s(u)=\left\| \nabla u \right\|_p^p+s[u]_{s,p}^p$. Since, $\{v_n\}\rightarrow 0$ in $L^t(\mathbb{R}^N)$ for all $t\in (p,p^*)$, by \eqref{HLS} we get $A(v_n)=o_n(1)$, thus by \eqref{5.9} we have
			\begin{eqnarray*}
				o_n(1) & = & M(u_n) =T_s(u_n)-\left\| u_n \right\|_{p^*}^{p^*}-\gamma_qA(u_n)\\
				& = & T_s(v_n) -\left\| v_n \right\|_{p^*}^{p^*} -A(v_n) +M(u)= T_s(v_n)-\left\| v_n \right\|_{p^*}^{p^*}+o_n(1),
			\end{eqnarray*}
			that is, $T_s(v_n)=\left\| v_n \right\|_{p^*}^{p^*}+o_n(1)$. Let $l=\displaystyle \lim_{n\rightarrow \infty}T_s(v_n)=\lim_{n\rightarrow \infty} \left\| v_n \right\|_{p^*}^{p^*}\geq 0$, then again by \eqref{S} we have, either $l=0$ or $l \geq S^{\frac{N}{p}}$. Now, if $l\geq S^{\frac{N}{p}}$, then by \eqref{5.9} and the fact that $A(v_n)=o_n(1)$, we will get
			\begin{eqnarray*}
				m_r(\tau) & = & \lim_{n\rightarrow \infty} E(u_n) = \lim_{n\rightarrow \infty} \left(\frac{T(u_n)}{p}-\frac{\left\| u_n \right\|_{p^*}^{p^*}}{p^*}-\frac{A(u_n)}{2q}\right)\\
				& = & \lim_{n\rightarrow \infty} \left(\frac{T(v_n)}{p}-\frac{\left\| v_n \right\|_{p^*}^{p^*}}{p^*}\right)+ E(u)\geq \lim_{n\rightarrow \infty} \left(\frac{T_s(v_n)}{p}-\frac{\left\| v_n \right\|_{p^*}^{p^*}}{p^*}\right)+ E(u)\\
				& = & \frac{l}{N}+E(u) \geq \frac{S^{\frac{N}{p}}}{N}+E(u).
			\end{eqnarray*}
			Hence \ref{Lemma 2.2.6a} holds true, and with $l=0$ we get strong convergence of $\{\left\| \nabla v_n\right\|_p^p\}$, $\{[v_n]_{s,p}^p\}$  and $\{\left\|u_n\right\|_{p^*}^{p^*}\}$. 
			Then $\displaystyle \lim_{n\rightarrow \infty}E'(u_n)(u_n)=E(u)(u)$, and hence
			\begin{eqnarray*}
				\lambda & = & \lim_{n\rightarrow \infty} \lambda_n =\lim_{n\rightarrow \infty}\frac{E'(u_n)(u_n)}{\tau}=\frac{E'(u)(u)}{\tau}\\
				& = & \frac{1}{\tau}\left(E'(u)(u)-M(u)\right)=\frac{(1-s)}{\tau}[u]_{s,p}^p+\frac{(\gamma_q-1)}{\tau}\mu \mathcal{A}_q(u)<0,
			\end{eqnarray*}
			for sufficiently large $\mu>0$. Moreover, by \eqref{5.9}, \eqref{E(u_n)} and \eqref{u} we get
			\begin{eqnarray*}
				\lambda \left\| v_n \right\|_p^p & = & \lambda_n \left\| u_n \right\|_p^p-\lambda \left\| u \right\|_p^p+o_n(1) = E'(u_n)(u_n)-E'(u)(u)+o_n(1)\\
				& = &-A(u_n)+A(u) +o_n(1) = -A(v_n)+o_n(1)= o_n(1).
			\end{eqnarray*} 
			Therefore $\{u_n\}\rightarrow u$ strongly in $W^{1,p}(\mathbb{R}^N)$ and hence \ref{Lemma 2.2.6b} holds.
		\end{proof}
		\begin{lemma}\label{Lemma 5.7}
			For all $\bar{q_s}\leq q <\frac{p}{2}\left(\frac{N+\alpha}{N-p}\right)$,
			\begin{equation}\label{m_r_tau}
				m_r(\tau)<\frac{S^{\frac{N}{p}}}{N}.
			\end{equation}
		\end{lemma}
		\begin{proof}
			Let $\psi\in C_c^{\infty}(\mathbb{R}^N, [0,1])$ be a radial cut-off function such that $\psi=1$ in $B_1(0)$ and $\psi=0$ in $\mathbb{R}^N\setminus B_2(0)$. For $\epsilon>0$, define $u_{\epsilon}(x):=\psi(x)U_{\epsilon}(x)$ where $U_{\epsilon}$ is as defined in \eqref{U_epsilon} with $x_0=0$, then by \cite{Avenia2015} we have:
			\begin{eqnarray}\label{gradient_u_epsilon}
				\left\| \nabla u_{\epsilon}\right\|_p^p = K_1+ O(\epsilon^{\frac{N-p}{p-1}}),
			\end{eqnarray}
			\begin{equation}\label{[u_epsilon]}
				[u_\epsilon]_{s,p}^{p} = O(\epsilon^{m_{N,p,s}}) \text{ where } m_{N,p,s}=\min\left\{\frac{N-p}{p-1},p(1-s)\right\},
			\end{equation}
			\begin{equation}\label{u_p*}
				\left\| u_\epsilon \right\|_{p^*}^{p^*} = K_2 +O(\epsilon^{-\frac{N}{p-1}}),
			\end{equation}
			with $S=K_1/K_2^{\frac{p}{p^*}}$. Since 
			$$\left\| u_\epsilon \right\|_p^p= \int_{B_1(0)}|U_{\epsilon}(x)|^pdx+ \int_{\mathbb{R}^N\setminus B_1(0)}|\psi (x)|^p|U_{\epsilon}(x)|^pdx= \int_{B_1(0)}|U_{\epsilon}(x)|dx+O(\epsilon^{\frac{N-p}{p-1}}),$$
			following the work of Brezis and Nirenberg in \cite{Brezis1983Positive}, we have:
			\begin{equation}\label{u_epsilon}
				\left\| u_{\epsilon} \right\|_p^p = \left\{
				\begin{array}{cc}
					K_3\epsilon^p+O(\epsilon^{\frac{N-p}{p-1}}), & \text{ for } N>p^2,\\
					K_4\epsilon^p|\ln(\epsilon)|+O(\epsilon^{\frac{N-p}{p-1}}) & \text{ for } N=p^2,\\
					K_5\epsilon^p+O(\epsilon^{\frac{N-p}{p-1}}) & \text{ for } p<N<p^2.
				\end{array}
				\right.
			\end{equation}
			Also, a direct computation give us:
			\begin{equation}\label{A(u_epsilon)}
				A(u_{\epsilon}) \geq K_6\epsilon^{2\bar{q_s}(1-\gamma_{\bar{q_s}})}.
			\end{equation}
			Set $v_{\epsilon}(x):=\frac{\tau^{\frac{1}{p}}u_{\epsilon}(x)}{\left\| u_{\epsilon}\right\|_p}$, clearly $v_\epsilon\in S_r(\tau)$ and hence by \autoref{Lemma 5.3}, there exists $t_{\epsilon}\in \mathbb{R}^N$ such that $t_{\epsilon}\star v_{\epsilon}\in P_r(\tau)$ and $I_{v_{\epsilon}}(t_{\epsilon})=\displaystyle\max_{t\in \mathbb{R}}I_{v_{\epsilon}}(t)$. 
			Considering the sequence $\{t_{\epsilon}\}$, one can observe that, if $\{t_{\epsilon}\}\rightarrow -\infty$ as $\epsilon\rightarrow 0$, then by \autoref{Lemma 5.1}, we must have 
			$$m_r(\tau)\leq E(t_{\epsilon}\star v_{\epsilon})\rightarrow 0^+ \text{ as }\epsilon\rightarrow 0,$$
			thus, $m_r(\tau)\leq 0 <\frac{S^{N/p}}{N}$. Similarly, if $\{t_{\epsilon}\}\rightarrow +\infty$ as $\epsilon \rightarrow 0$, using \autoref{Lemma 5.1}, we get, $m_r(\tau)<0<\frac{S^{N/p}}{N}$. In fact, by \autoref{Lemma 5.4}, we have $m_r(\tau)=r(\tau)>0$, thus the above cases are not possible, that is, $\{t_{\epsilon}\}$ must be bounded.\\
			Now, let $a<b$, be such that such that $a<t_{\epsilon}<b$ for all $\epsilon>0$. Since $M(t_{\epsilon}\star v_{\epsilon})=0$, we have 
			\begin{equation*}
				e^{(p^*-p)t_{\epsilon}}\left\| v_{\epsilon}\right\|_{p^*}^{p^*} = B_{\epsilon}(v_{\epsilon})-\gamma_qe^{(2q\gamma_q-p)t_{\epsilon}}A(v_{\epsilon}),
			\end{equation*}
			where $B_{\epsilon}(u)=\left\| \nabla u \right\|_p^p+se^{(sp-p)t_{\epsilon}}[u]_{s,p}^p$. Thus, we have:
			\begin{equation}\label{e_t_epsilon}
				e^{pt_{\epsilon}}\leq \left(\frac{B_{\epsilon}(v_{\epsilon})}{\left\| v_{\epsilon}\right\|_{p^*}^{p^*}}\right)^{\frac{p}{p^*-p}}=
				\frac{\left\| u_{\epsilon}\right\|_p^p}{\tau}\left(\frac{B_{\epsilon}(u_{\epsilon})}{\left\| u_{\epsilon}\right\|_{p^*}^{p^*}}\right)^{\frac{p}{p^*-p}}.
			\end{equation} 
			Now, define $$g(t):=\frac{e^{pt}\left\| \nabla v_{\epsilon}\right\|_p^p}{p}-\frac{e^{p^*t}\left\| v_{\epsilon}\right\|_{p^*}^{p^*}}{p^*},$$
			clearly $g$ has global maxima at $t=t_0$ such that
			$$e^{t_0}=\left(\frac{\left\| \nabla v_{\epsilon}\right\|_p^p}{\left\| v_{\epsilon}\right\|_{p^*}^{p^*}}\right)^{\frac{1}{p^*-p}},$$
			with
			\begin{eqnarray}\label{g(t_0)}
				\max_{t\in \mathbb{R}}g(t) & = & g(t_{0})=\frac{1}{N}\left(\frac{\left\| \nabla v_{\epsilon}\right\|_p^p}{\left\| v_{\epsilon}\right\|_{p^*}^p}\right)^{\frac{p^*}{p^*-p}}= \frac{1}{N}\left(\frac{\left\| \nabla u_{\epsilon}\right\|_p^p}{\left\| u_{\epsilon}\right\|_{p^*}^p}\right)^{\frac{p^*}{p^*-p}}\nonumber\\
				& = & \frac{1}{N}S^{\frac{N}{p}}+O(\epsilon^{\frac{N-p}{p-1}}),
			\end{eqnarray}
			by \eqref{gradient_u_epsilon} and \eqref{u_p*}. This gives us:
			\begin{eqnarray}\label{E(t_epsilon)}
				E(t_{\epsilon}\star v_{\epsilon}) & = & g(t_{\epsilon})+\frac{e^{spt_{\epsilon}}[v_{\epsilon}]_{s,p}^p}{p}-\frac{e^{2q\gamma_q t_{\epsilon}}A(v_{\epsilon})}{2q}\nonumber\\
				& \leq & \frac{1}{N}S^{\frac{N}{p}}+O(\epsilon^{\frac{N-p}{p-1}})+\frac{e^{spt_{\epsilon}}[v_{\epsilon}]_{s,p}^p}{p}-\frac{e^{2q\gamma_q t_{\epsilon}}A(v_{\epsilon})}{2q}.
			\end{eqnarray}
			Now, for $N>p^2$, by \eqref{gradient_u_epsilon}-\eqref{A(u_epsilon)} and \eqref{e_t_epsilon} we have:
			\begin{eqnarray*}
				\frac{e^{pt_{\epsilon}}[v_{\epsilon}]_{s,p}^p}{A(v_{\epsilon})} & = & \frac{e^{pt_{\epsilon}}\left\| u_{\epsilon}\right\|_p^{2q-p}[u_{\epsilon}]_{s,p}^p}{\tau^{\frac{2q-p}{p}}A(u_{\epsilon})} \leq \left(\frac{e^{pt_{\epsilon}}\left\| u_{\epsilon}\right\|_p^{2q-p}[u_{\epsilon}]_{s,p}^p}{\tau^{\frac{2q-p}{p}}A(u_{\epsilon})}\right) \left(\frac{B_{\epsilon}(u_{\epsilon})}{\left\| u_{\epsilon}\right\|_{p^*}^{p^*}}\right)^{\frac{p}{p^*-p}}\frac{\left\| u_{\epsilon}\right\|_p^p}{\tau}\\
				& = & \frac{\left\| u_{\epsilon}\right\|_p^{2q}[u_{\epsilon}]_{s,p}^p B_{\epsilon}(u_{\epsilon})^{\frac{p}{p^*-p}}}{\tau^{2q}A(u_{\epsilon})(\left\| u_{\epsilon}\right\|_{p^*}^{p^*})^{\frac{p}{p^*-p}}}\\
				& \leq & \frac{K'\epsilon^{(m_{N,p,s}+2q\gamma_q)}(K_3+O(\epsilon^{\frac{N-p}{p-1}}))^{2q}(K_1+O(\epsilon^{\frac{N-p}{p-1}}+O(\epsilon^{m_{N,ps}}))^{\frac{p}{p^*-p}}}{\tau^{2q}(K_2+O(\epsilon^{-\frac{N}{p-1}}))^{\frac{p}{p^*-p}}}\\
				& & \rightarrow 0 \text{ as } \epsilon \rightarrow 0.
			\end{eqnarray*}
			Therefore, $\frac{[v_{\epsilon}]_{s,p}^p}{A(v_{\epsilon})}\rightarrow 0$ as $\epsilon\rightarrow 0$, since, $0<e^{ap}<e^{pt_{\epsilon}}<e^{bp}$. 
			Thus, we can find $\epsilon_0>0$ small enough, such that $$\frac{e^{spt_{\epsilon}}[v_{\epsilon}]_{s,p}^p}{p}-\frac{e^{2q\gamma_q t_{\epsilon}}A(v_{\epsilon})}{2q}<0 \text{ for all }\epsilon<\epsilon_0.$$
			Hence, by \eqref{E(t_epsilon)}, we can conclude \eqref{m_r_tau}. Following the same procedure, one can easily deduce \eqref{m_r_tau} for $p<N\leq p^2$.
		\end{proof}
		\begin{myproof}{Theorem}{\ref{Theorem 1.3}}
			Let $\{u_n\}$ be the sequence deduced in \autoref{Lemma 5.5}. Suppose \ref{Lemma 2.2.6a} of \autoref{Lemma 5.6} holds, then there exists $u\in W^{1,p}(\mathbb{R}^N)$ such that $\{u_n\}\rightharpoonup u$ in $W^{1,p}(\mathbb{R}^N)$, and $u$ solves \eqref{1.1} with 
			$$E(u)\leq m_r(\tau)-\frac{S^{N/p}}{N}<0,$$
			by \autoref{Lemma 5.7}. Thus, by \autoref{Lemma 5.1} and \autoref{Lemma 5.3} we have, $I'_{u}(0)<0$ and hence $M(u)=M(0\star u)=I'_{u}(0)<0$. But since $u$ solve \eqref{1.1}, we must have $M(u)=0$. Therefore, \ref{Lemma 2.2.6b} of \autoref{Lemma 5.6} holds and hence the required result.
		\end{myproof}
		

\begin{thebibliography}{100}
			\bibitem{Avenia2015} P. d'Avenia, G Siciliano and M. Squassina, {\it On fractional Choquard equations}, Mathematical Models and Methods in Applied Sciences 25(2015), no. 08, 1447-76.
				\bibitem{bartsch2012normalized}
			T. Bartsch and S. de Valeriola, {\it Normalized solutions of nonlinear Schrödinger equations}, Arch. Math. 100 (2013), 75-83.
			\bibitem{bartsch2018normalized} 
			T. Bartsch and L. Jeanjean, {\it Normalized solutions for nonlinear Schrödinger systems}, Proceedings of the Royal Society of Edinburgh Section A: Mathematics 148 (2018), no. 2, 225-242.
			\bibitem{bartsch2016normalized}
			T. Bartsch, L. Jeanjean and N. Soave, {\it  Normalized solutions for a system of coupled cubic Schrödinger equations on $\mathbb{R}^3$}, Journal de Mathématiques Pures et Appliquées 106 (2016), no. 4, 583-614.
			\bibitem{bartsch2019multiple} 
			T. Bartsch and N. Soave, {\it Multiple normalized solutions for a competing system of Schrödinger equations}, Calculus of Variations and Partial Differential Equations 58 (2019), 1-24.
			\bibitem{Biswas2023Regularity} R. Biswas and S. Tiwari, {\it Regularity results for Choquard equations involving fractional p‐Laplacian}, Mathematische Nachrichten 296(2023), no. 09, 4060-85.
			\bibitem{Brezis1983Positive} H. Brézis and L. Nirenberg, {\it Positive solutions of nonlinear elliptic equations involving critical Sobolev exponents}, Communications on pure and applied mathematics 36(1983), no. 4, 437-77.
			\bibitem{filippucci2020singular} 
			R. Filippucci and M. Ghergu, {\it Singular solutions for coercive quasilinear elliptic inequalities with nonlocal terms}, Nonlinear Analysis 197 (2020), 111857.
			\bibitem{Fiorenza2021Gagliardo} A. Fiorenza, M.R. Formica, T.G. Roskovec and F. Soudský, {\it Detailed proof of classical Gagliardo–Nirenberg interpolation inequality with historical remarks}, Zeitschrift für Analysis und ihre Anwendungen 40(2021), no. 02, 217-36.
			\bibitem{Garain2023Higher} P. Garain and E. Lindgren {\it Higher Hölder regularity for mixed local and nonlocal degenerate elliptic equations}, Calculus of Variations and Partial Differential Equations 62(2023), no. 2, 67.
			\bibitem{Nidhi2023} J. Giacomoni, N. Nidhi and K. Sreenadh, {\it Normalized solution to a Choquard equation involving mixed local and non	local operators}, preprint.
			\bibitem{Giacomoni2024Normalized} J. Giacomoni, N. Nidhi and K. Sreenadh, {\it Normalized solutions to a critical growth Choquard equation involving mixed operators}, Asymptotic Analysis 2024(Preprint), 1-34.
			 \bibitem{gou2018multiple} 
			T. Gou and L. Jeanjean {\it Multiple positive normalized solutions for nonlinear Schrödinger systems}, Nonlinearity 31 (2018), no. 5, 2319.
			\bibitem{Jeanjean1997Existence} L. Jeanjean, {\it Existence of solutions with prescribed norm for semilinear elliptic equations}, Nonlinear Analysis: Theory, Methods and Applications 28(1997), no. 10, 1633-59.
			\bibitem{Kesavan2019Topics} S. Kesavan, {\it Topics in Functional Analysis and Applications}, New Age International Publishers, New Delhi, vol. 23 (2019).
			\bibitem{lei2023sufficient} 
			C. Lei, M. Yang and B. Zhang, {\it Sufficient and Necessary Conditions for Normalized Solutions to a Choquard Equation}, The Journal of Geometric Analysis 33(2023), no. 4, 109.
			   \bibitem{lieb1977existence} 
			E.H. Lieb, {\it Existence and uniqueness of the minimizing solution of Choquard's nonlinear equation}, Studies in Applied Mathematics 57 (1977), no. 2, 93-105.              
			\bibitem{Lions1984concentration} P.L. Lions, {\it The concentration-compactness principle in the calculus of variations. The locally compact case, part 2}, InAnnales de l'Institut Henri Poincaré C, Analyse non linéaire 1(1984), No. 4, 223-283. 
			  	\bibitem{liu2022another} 
			Z. Liu, V.D. R\u adulescu, C. Tang and J. Zhang, {\it Another look at planar Schrödinger-Newton systems}, Journal of Differential Equations 328 (2022), 65-104.  
			\bibitem{Meng2024Normalized} Y. Meng and X. He, {\it Normalized Solutions for the Fractional Choquard Equations with Hardy–Littlewood–Sobolev Upper Critical Exponent}, Qualitative Theory of Dynamical Systems 23(2024), no. 1, 19.              	
			\bibitem{moroz2013groundstates}
			V. Moroz and J. Van Schaftingen, {\it Groundstates of nonlinear Choquard equations: existence, qualitative properties and decay asymptotics}, Journal of Functional Analysis 265 (2013), no. 2, 153-184.
			\bibitem{Moroz2015} V Moroz and J. Van Schaftingen, {\it Existence of groundstates for a class of nonlinear Choquard equations}, Transactions of the American Mathematical Society 367(2015), no. 9, 6557-79.				
			\bibitem{Moroz2017Guide}V. Moroz and J. Van Schaftingen, {\it A guide to the Choquard equation}, Journal of Fixed Point Theory and Applications 19(2017), 773-813.
				\bibitem{noris2014stable}
			B. Noris, H. Tavares and G. Verzini, {\it Stable solitary waves with prescribed $ L^ 2$-mass for the cubic Schrödinger system with trapping potentials}, Discrete and Continuous Dynamical Systems 35 (2015), no. 12, 6085-6112.
			\bibitem{noris2015existence}
			B. Noris, H. Tavares and G. Verzini, {\it Existence and orbital stability of the ground states with prescribed mass for the $L^2$-critical and supercritical NLS on bounded domains}, Analysis and PDE 7 (2015), no. 8, 1807-1838.
			\bibitem{noris2019normalized}
			B. Noris, H. Tavares and G. Verzini, {\it Normalized solutions for nonlinear Schrödinger systems on bounded domains}, Nonlinearity 32 (2019), no. 3, 1044. 
			\bibitem{Park2011FractionalGN}Y.J. Park, {\it Fractional Gagliardo-Nirenberg inequality}, Journal of the Chungcheong Mathematical Society 24(2011), no.3, 583-586.
			 \bibitem{pekar1954untersuchungen} S.I.Pekar, {\it Untersuchungen {\"u}ber die Elektronentheorie der Kristalle}, De Gruyter(1954). 
			  \bibitem{pellacci2021normalized} B. Pellacci , A. Pistoia, G. Vaira and G. Verzini, {\it Normalized concentrating solutions to nonlinear elliptic problems}, Journal of Differential Equations 275 (2021), 882-919.
			    \bibitem{penrose1996gravity} R. Penrose, {\it On gravity's role in quantum state reduction}, General relativity and gravitation 28 (1996), 581-600.  
			  \bibitem{pierotti2017normalized} D. Pierotti and G. Verzini, {\it Normalized bound states for the nonlinear Schrödinger equation in bounded domains}, Calculus of Variations and Partial Differential Equations 56 (2017), 1-27.
			  \bibitem{Shang2023Normalized} X. Shang and P. Ma, {\it Normalized solutions to the nonlinear Choquard equations with Hardy-Littlewood-Sobolev upper critical exponent}, Journal of Mathematical Analysis and Applications 521(2023), no. 2, 126916.
			   \bibitem{Shen2024Normalized} X. Shen, Y. Lv and Z. Ou, {\it Normalized Solutions to the Fractional Schr{\"o}dinger Equation with Critical Growth}, Qualitative Theory of Dynamical Systems 23 (2024), no. 3, 145.
			\bibitem{Yuan2013radial} W. Sickel, D. Yang and W. Yuan, {\it The radial lemma of Strauss in the context of Morrey spaces}. Annales Fennici Mathematici 39 (2014), 417–442. 
			 \bibitem{Soave2020Normalized} N. Soave, {\it Normalized ground states for the NLS equation with combined nonlinearities}, Journal of Differential Equations 269 (2020), no. 9, 6941-6987.
			\bibitem{Talenti1976} G. Talenti, {\it Best constant in Sobolev inequality}, Annali di Matematica pura ed Applicata. 110(1976), 353-372.
		\end{thebibliography}
	\end{document}